\documentclass[11pt,a4paper]{article}
\usepackage{amsmath}
\usepackage{amsfonts}
\usepackage{amssymb}
\usepackage{amsthm}
\usepackage{mathrsfs}
\usepackage{dsfont}
\usepackage{paralist}
\usepackage{natbib}
\usepackage{bm}
\usepackage{color}
\usepackage[colorlinks=true,linkcolor=blue,citecolor=blue,pdfborder={0 0 0}]{hyperref}
\usepackage{graphicx}
\usepackage{tabularx}
\usepackage[left=3.5cm,right=3.5cm,top=2cm,bottom=2cm,includeheadfoot]{geometry}
\usepackage{comment}

\newcount\Comments  % 0 suppresses notes to selves in text
\newcommand{\kibitz}[2]{\ifnum\Comments=1\textcolor{#1}{#2}\fi}

\newtheoremstyle{normal}% name
{2ex}               % Space above, empty = `usual value'
{3ex}               % Space below
{}                  % Body font
{}                  % Indent amount (empty = no indent, \parindent = para indent)
{\bfseries} % Thm head font
{}                  % Punctuation after thm head
{2pt}   % Space after thm head.
{\thmname{#1}\thmnumber{ #2.} \thmnote{(#3)}}% Thm head spec

\newtheoremstyle{italic}% name
{2ex}%      Space above, empty = `usual value'
{3ex}%      Space below
{\itshape}% Body font
{}%         Indent amount (empty = no indent, \parindent = para indent)
{\bfseries} % Thm head font
{}%        Punctuation after thm head
{2pt}%     Space after thm head.
{\thmname{#1}\thmnumber{ #2.} \thmnote{(#3)}}% Thm head spec

\theoremstyle{normal}
\newtheorem{definition}{Definition}[section]
\newtheorem{remark}[definition]{Remark}

\newtheorem{condition}[definition]{Condition}

\theoremstyle{italic}
\newtheorem{theorem}[definition]{Theorem}
\newtheorem{lemma}[definition]{Lemma}
\newtheorem{proposition}[definition]{Proposition}

\newcommand\N{\mathbb{N}}
\newcommand\Q{\mathbb{Q}}
\newcommand\R{\mathbb{R}}

\newcommand\Prob{\mathbb{P}}    %probability
\newcommand\Bb{\mathbb{B}}

\newcommand\Eb{\mathbb{E}}

\newcommand\Gb{\mathbb{G}}

\newcommand\ovw{\overline{w}}
\newcommand\ovr{\overline{r}}
\newcommand\ovv{\overline{v}}

\newcommand\weak{\rightsquigarrow}

\newcommand\defeq{:=}

\newcommand\Deli{\Delta_i^n}
\newcommand\Delj{\Delta_j^n}
\newcommand\Indit{\mathtt 1_{(- \infty, t ]}}
\newcommand\Truniv{\mathtt 1_{\lbrace | \Delta_i^n X | > v_n \rbrace}}
\newcommand\Trunx{\mathtt 1_{\lbrace |x| > v_n \rbrace}}
\newcommand\linfr{\ell^{\infty}(\R)}
\newcommand\tibigj{\tilde X^{\prime \prime n}}
\newcommand\tibigjal{\tilde X^{\prime \prime}(\alpha)^n}
\newcommand\tibigjeial{\tilde X^{\prime \prime}(8 \alpha)^n}
\newcommand\hatbigjal{\hat X^{\prime \prime}(\alpha)}

\allowdisplaybreaks[1]

\begin{document}

\title{Weak convergence of the empirical truncated distribution function of the L\'evy measure of an It\=o semimartingale}

\author{Michael Hoffmann\footnotemark[1]\  ~and  Mathias Vetter\footnotemark[2], \bigskip \\
{ Ruhr-Universit\"at Bochum \ \& \ Christian-Albrechts-Universit\"at zu Kiel}
}

\footnotetext[1]{Ruhr-Universit\"at Bochum,
Fakult\"at f\"ur Mathematik, 44780 Bochum, Germany.
{E-mail:} michael.hoffmann@rub.de}

\footnotetext[2]{Christian-Albrechts-Universit\"at zu Kiel, Mathematisches Seminar, Ludewig-Meyn-Str.\ 4, 24118 Kiel, Germany.
{E-mail:} vetter@math.uni-kiel.de}

\maketitle

\begin{abstract}
Given an It\=o semimartingale with a time-homogeneous jump part observed at high frequency, we prove weak convergence of a normalized truncated empirical distribution function of the L\'evy measure to a Gaussian process. In contrast to competing procedures, our estimator works for processes with a non-vanishing diffusion component and under simple assumptions on the jump process.  
\end{abstract}

%\medskip

 \textit{Keywords and Phrases:}  Empirical distribution function; high-frequency statistics; It\=o semimartingale; L\'evy measure; weak convergence

\smallskip

 \textit{AMS Subject Classification:} 60F17, 60G51 (primary); 62G99, 62M99 (secondary)

%\tableofcontents

\section{Introduction}
\def\theequation{1.\arabic{equation}}
\setcounter{equation}{0}

Recent years have seen a growing interest in statistical methods for time-continuous processes exhibiting jumps, in particular for L\'evy processes and related models, as these processes possess a rather simple mathematical form but allow for a flexible modelling of various real-life phenomena. In the framework of discrete observations of such processes, two different strands of literature have been developed, depending on the nature of the sampling scheme and its asymptotics. Let $n$ denote the number of observations and $\Delta_n > 0$ the distance between two successive observations of the underlying process $X$. Besides the natural assumption $n \Delta_n \to \infty$ of a growing time horizon, which in general cannot be avoided due to the fact that only finitely many large jumps exist over any finite interval, one has to distinguish between low-frequency observations with $\Delta = \Delta_n$ being fixed and high-frequency observations with $\Delta_n \to 0$ as well. 

Usually, the statistical methods are highly different in both contexts, and it is well-known that not all characteristics of a L\'evy process, say, can be recovered in both situations. In the low-frequency situation, the focus is typically on methods from the frequency domain and involves estimation of the characteristic exponent of $X$ in order to identify the quantities of interest. See e.g.\ \cite{NeuRei09}, \cite{Gug12} or \cite{NicRei12}. On the other hand, for high-frequency observations one mostly remains in the time domain and identifies e.g.\ jumps of $X$ from large increments over small intervals. With a view on the L\'evy measure, this approach has been used for instance in \cite{FigLop08} or \cite{Hof14}. 

Most recently, \cite{Nic15} presented several approaches to estimate
\[
N_{\rho}(t) = \int \limits_{-\infty}^t \rho(x) \nu(dx),
\]
where $\nu$ denotes the L\'evy measure and $\rho$ is chosen appropriately such that the integral is always defined. Under weak conditions on $\rho$, this L\'evy distribution function determines the entire jump behaviour of $X$, just like probability measures being determined by standard distribution functions. Among other estimators, including a spectral estimator in the spirit of \cite{NicRei12}, the authors discuss properties of the natural estimator from the high-frequency framework, which counts increments of $X$ below the threshold $t$ and weights them according to $\rho$. Precisely, they use
\[
N_\rho^{(n)}(t) = \frac{1}{n \Delta_n} \sum \limits_{i = 1}^n \rho(\Deli X) \Indit(\Deli X), 
\]
where $\Deli X \defeq X_{i \Delta_n} - X_{(i-1) \Delta_n}$ denotes the increment of $X$ over $[(i-1)\Delta_n, i\Delta_n]$. The authors establish weak convergence of $\sqrt{n \Delta_n} \big(N_\rho^{(n)}(t) - N_\rho(t)\big)$ to a Gaussian process, but only for L\'evy processes without a diffusion component and under additional conditions on the L\'evy measure of which some are difficult to check. 

Given the need to assess the unknown L\'evy measure for various applications like model validation or to identify changes in the temporal behaviour, it is unsatisfactory that estimators in the time domain only work when no Brownian component is present. For this reason, we propose a natural extension using a truncation technique which allows for limit theorems involving diffusion components as well, that is 
\[
\overline N_{\rho}^{(n)}(t) = \frac{1}{n \Delta_n} \sum \limits_{i=1}^n \rho(\Deli X) \Indit(\Deli X) \Truniv
\]
for a suitable sequence $v_n$. Truncation methods in the high-frequency regime date back to \cite{mancini} and have usually been used to cut off jumps in order to focus on continuous movements of the process only. Here, we use truncation to identify jumps, which is crucial to estimate the L\'evy distribution function around zero correctly. Moreover, we allow the continuous part of $X$ to be a general It\=o semimartingale, and our conditions on the jump measure are sufficiently general to accommodate a variety of well-known jump processes from the literature as well. 

In the following, $X$ denotes an It\=o semimartingale with characteristics $(b_s, \sigma_s, \overline \mu)$, that is a stochastic process with the decomposition
\begin{multline}
\label{ItoSemimart}
X_t = X_0 + \int \limits_0^t b_s ds + \int \limits_0^t \sigma_s dW_s + \int \limits_0^t \int \limits_{\R} x \mathtt 1_{\lbrace |x| \leq 1 \rbrace} (\mu - \overline \mu)(ds,dx) \\
+ \int \limits_0^t \int \limits_{\R} x \mathtt 1_{\lbrace |x| >1 \rbrace} \mu(ds,dx).
\end{multline}
Here, $b_s$ and $\sigma_s$ are predictable processes from $\Omega \times \R_+$ to $\R$, $W$ denotes a standard Brownian motion and $\mu$ is the random measure associated with the jumps of $X$. We assume that the jump behaviour of $X$ is constant over time, in which case the compensator $\overline \mu$ of $\mu$ is of the form $\overline \mu(ds,dx) = ds \nu(dx)$, where the L\'evy measure $\nu(dx)$ integrates the function $1 \wedge x^2$ and does not charge $0$. Observations come regularly in a high-frequency regime, i.e.\ at stage $n \in \N$ we observe realizations of $X$ at the equidistant time points $i \Delta_n$ with $i = 0, \ldots ,n$, where the mesh $\Delta_n \rightarrow 0$, while $n \Delta_n \rightarrow \infty$. 

The remainder of the paper is organised as follows: Section \ref{sec:cond} deals with the conditions on the process and the auxiliary sequences, which we need in order for weak convergence of $\sqrt{n \Delta_n} \big(\overline N_{\rho}^{(n)}(t) - N_\rho(t)\big)$ to hold. These assumptions are rather mild and satisfied by a number of standard models. Section \ref{sec:converge} contains the main theorems of this work, as well as a short overview on the strategy we use in order to establish these results. All proofs are gathered in an Appendix which is Section \ref{sec:app}.

\section{Conditions on the underlying process and the estimator} \label{sec:cond}
\def\theequation{2.\arabic{equation}}
\setcounter{equation}{0}

Suppose we have complete knowledge of the distribution function $N_{\rho}(t)$ for a function $\rho$ which satisfies $0 \leq \rho(x) \leq K(1 \wedge x^2)$ for some constant $K>0$ and $\rho(x) >0$ for $x \neq 0$. Obviously, the measure with density $M(dx) \defeq \rho(x) \nu(dx)$ is completely determined from knowledge of the entire function $N_{\rho}$ and does not charge zero. Therefore, $1/\rho(x) M(dx) = \nu(dx)$ and consequently the jump behaviour of the It\=o semimartingale is known as well. For all possible applications it is thus sufficient to draw inference on $N_{\rho}$ only.  

Throughout this work we assume that $X$ is defined on a filtered probability space $(\Omega, \mathcal F, \Prob, (\mathcal F_t)_{t \in \R_+})$ and has a representation as in \eqref{ItoSemimart}. Recall further that at stage $n \in \N$ we observe realisations of $X$ at the equidistant time points $i \Delta_n$ with $0 \leq i \leq n$. In order to establish weak convergence of the estimator $\overline N^{(n)}_{\rho}(t)$, we state some further conditions on the underlying process and the auxiliary variables. 

\begin{condition}
\label{EasierCond}
Let $0< \beta < 2$ and $0< \zeta < \tau < 1/16$. Furthermore define $p \defeq 8(1+3 \beta) \frac{1+ \tau}{1-16 \tau}$. \\

\begin{compactenum}[(a)]
\item \textit{Conditions on the L\'evy measure and the function $\rho$:} \\
      The L\'evy measure has a Lebesgue density $h$ which satisfies:
			\begin{enumerate}[(1)]
			\item $h(x) \leq K|x|^{-(1+ \beta)}$ for $x$ in a neighbourhood of $0$ and some $K>0$.
			\item $h(x)$ is bounded on each $C_n \defeq \lbrace x \colon \frac{1}{n} \leq |x| \leq n \rbrace$ with $n \in \N$.
			      \label{DiCondmiddle}
			\item There is an $M>0$ such that $h(x) \leq K |x|^{-p- \epsilon}$ for some $K > 0$, when $|x| \geq M$ with some $\epsilon >0$.
						\label{DiCondinfty}
      \item $\rho \colon \R \rightarrow \R$ is a bounded $\mathcal C^1$-function with $\rho(0)=0$ and its derivative satisfies 
			      $|\rho^{\prime}(x)| \leq K |x|^{p-1}$ for all $x \in \R$ and some constant $K>0$.
			\end{enumerate}
\item \textit{Conditions on the truncation sequence $v_n$ and the observation scheme:} \\
     The truncation sequence $v_n$ satisfies
		 $$v_n \defeq \gamma \Delta_n^{\ovw},$$
		 with $\ovw = 1/8$ and some $\gamma >0$. Define further:
		 \begin{align*}
		 t_1 \defeq  ( 1+ \tau )^{-1} \quad \text{and } \quad
		 t_2 \defeq  ( 1+ \zeta )^{-1}.
		 \end{align*}
		 Then we have $0 < t_1 < t_2 < 1$ and we suppose that the observation scheme satisfies
		 $$\Delta_n = o(n^{-t_1}) \quad \text{ and } \quad n^{-t_2} = o(\Delta_n).$$ 
\item \label{DriDiffMomCond} \textit{Conditions on the drift and the diffusion coefficient:} \\
      For the function
			$$A(\omega) \defeq \left\{ \sup \limits_{s \in \R} \left| b_s(\omega) \right| \vee \sup \limits_{s \in \R} 
			\left| \sigma_s(\omega) \right| \right\},$$
			on $\Omega$ we have
			$$\Eb^{\ast} A^m < \infty,$$
			with 
			$$m = \bigg ( \left \lfloor \frac{8 + 7 \beta - \beta^2}{3 - \beta} \right \rfloor +1 \bigg ) \vee 4 \in \lbrace 4, \ldots, 18
			\rbrace,$$
			where $\Eb^{\ast}$ denotes outer expectation and $\lfloor z \rfloor$ is the largest integer smaller or equal to $z$. \qed
\end{compactenum}
\end{condition}

\begin{remark}
While Condition \ref{EasierCond} (c) is extremely mild, as it requires only a bound on the moments of drift and volatility, the two other assumptions are more restrictive: 

Part (a) basically says that the L\'evy measure has a continuous L\'evy density, which behaves near zero like the one of a $\beta$-stable process, whereas it has to decay sufficiently fast at infinity. Such conditions are well-known in the literature and often used in similar works on high-frequency statistics; see e.g.\ \cite{AitJac09} or \cite{AitJac10}. Common models in finance like the variance gamma process for the log stock price also satisfy our assumptions (see for instance \cite{Mad98}). Also, the function 
\[
\tilde \rho(x) = \begin{cases}
			                        0, \quad &\text{if } x=0 \\
															e^{-1/|x|}, \quad &\text{if } |x| >0 \\
															\end{cases}
\]
is suitable for any choice of the constants $\beta$ and $\tau$. In practice, however, one would like to work with a polynomial decay at zero, in which case the condition on $p$ comes into play. Here, the smaller $\beta$ and $\tau$, the smaller $p$ can be chosen.

Besides conditions on $X$ and $\rho$, it is crucial to choose the observation scheme in a specific manner. Obviously, $\Delta_n \rightarrow 0 $ and $n \Delta_n \rightarrow \infty$ because of $0< t_1 < t_2 <1$, and one would typically pick $\Delta_n = O(n^{-y})$ and $n^{-y} = O(\Delta_n)$ for some $0< t_1 < y < t_2 <1$. \qed
\end{remark}

It is possible to work with even weaker assumptions, as can be seen from Condition \ref{Cond1} and Proposition \ref{easier} in the Appendix. Nevertheless, for the ease of exposition we stick to the set of assumptions above which are much simpler to check and to interpret.

\section{Convergence of the truncated empirical distribution function of the L\'evy measure} \label{sec:converge}
\def\theequation{3.\arabic{equation}}
\setcounter{equation}{0}

Recall from the introduction that, for a suitable function $\rho$, we consider the truncated empirical distribution functions of the L\'evy measure, which are defined as 
$$\overline N_{\rho}^{(n)}(t) = \frac{1}{n \Delta_n} \sum \limits_{i=1}^n \rho(\Deli X) \Indit(\Deli X) \Truniv.$$
These quantities can be considered as estimators for the distribution function 
$$N_{\rho}(t) = \int \limits_{- \infty}^t \rho(x) \nu(dx)$$ 
at the point $t \in \R$. Furthermore, we define the empirical processes 
$$G_{\rho}^{(n)}(t) = \sqrt{n \Delta_n}\big(\overline N_{\rho}^{(n)}(t) - N_{\rho}(t)\big).$$
Below, we state our main result.

\begin{theorem}
\label{ConvThm}
Let $X$ be an It\=o semimartingale and let $\rho \colon \R \rightarrow \R$ be a $\mathcal C^1$ function such that Condition \ref{EasierCond} is satisfied. Suppose further that the observation scheme meets the properties of Condition \ref{EasierCond}. Then we have the weak convergence
$$G_{\rho}^{(n)} \weak \Gb_{\rho}$$
in $\linfr$, where $\Gb_{\rho}$ is a tight mean zero Gaussian process in $\linfr$ with covariance function
$$H_{\rho}(u,v) \defeq \int \rho^2(x) \mathtt 1_{(- \infty, u \wedge v]}(x) \nu(dx).$$
Additionally, the sample paths of $\Gb_{\rho}$ are almost surely uniformly continuous with respect to the semimetric
$$d_{\rho}(u,v) = \left\{ \int \rho^2(x) \mathtt 1_{(u \wedge v, u \vee v]}(x) \nu(dx) \right\}^{1/2}.$$
\end{theorem}

\begin{remark}
\label{RescBMRem}
Set $c_\rho = \int \rho^2(x) \nu(dx)$ and consider a standard Brownian motion $\Bb$ on 
$[0,c_\rho]$. It is a well known fact (see for instance Section 8 in \cite{Bil99}) that the law of this process is tight in $\ell^{\infty}([0,c_\rho])$. Furthermore, the sample paths of $\Bb$ are uniformly continuous with respect to the Euclidean distance and for each $\epsilon, \eta >0$ there is a $\delta >0$ such that
\begin{eqnarray}
\label{EquicontAbsch}
\Prob \big ( \sup \limits_{|u-v|^{1/2} < \delta} \left| \Bb (u) - \Bb(v) \right| > \epsilon \big ) < \eta.
\end{eqnarray}
This is a consequence of Theorem 1.5.7 and Addendum 1.5.8 in \cite{VanWel96}. 

Because of Lemma 1.3.12(ii) in the previously mentioned reference two tight Borel laws on $\ell^{\infty}(T)$ (for an arbitrary set $T$) are equal if they have the same marginal distributions. Therefore the limit distribution of Theorem \ref{ConvThm} is equal to the law of the rescaled Brownian motion
$$\Bb_\rho(t) = \Bb \bigg ( \int \rho^2(x) \mathtt 1_{(- \infty, t]}(x) \nu(dx) \bigg ),$$
because the latter process is in fact tight in $\linfr$ by \eqref{EquicontAbsch} and Theorem 1.5.6 in \cite{VanWel96}. The sample paths of $\Bb_\rho$ are also uniformly continuous with respect to $d_\rho$. \qed
\end{remark}

Let us sketch the main idea behind the proof of Theorem \ref{ConvThm}. We choose an auxiliary function $\Psi \colon \R_+ \rightarrow \R$, which is $\mathcal C^{\infty}$ and satisfies $\mathtt 1_{[1, \infty)}(x) \leq \Psi(x) \leq \mathtt 1_{[1/2, \infty)}(x)$ for all $x \in \R_+$. For $\alpha >0$ define $\Psi_{\alpha} \colon \R \rightarrow \R$ through $\Psi_{\alpha}(x) = \Psi(|x|/\alpha)$ and let $\Psi_{\alpha}^{\prime} \colon \R \rightarrow \R$ be the function $\Psi_{\alpha}^{\prime}(x) = 1 - \Psi_{\alpha}(x)$. These functions are used to distinguish between small and large increments of $X$ which need different treatments. 

For the function $\rho$ we define $\rho_{\alpha}(x) = \rho(x) \Psi_{\alpha}(x)$ and $\rho_{\alpha}^{\prime}(x) = \rho(x) \Psi^{\prime}_{\alpha}(x)$. Furthermore, let 
$$g_t^{(\alpha)}(x) = \rho(x) \Psi_{\alpha}(x) \Indit(x) \quad \text{ and } \quad g_t^{\prime (\alpha)} (x) = 
\rho(x) \Psi^{\prime}_{\alpha}(x) \Indit(x),$$
for $x,t \in \R$ and define the following empirical processes:
\begin{eqnarray*}
&G_{\rho,n}^{(\alpha)}(t) = \sqrt{n \Delta_n} \left\{ \frac{1}{n \Delta_n} \sum \limits_{i=1}^n g_t^{(\alpha)}(\Deli X) \Truniv 
- N_{\rho_{\alpha}}(t) \right\}, \\
&G_{\rho,n}^{\prime (\alpha)}(t) = \sqrt{n \Delta_n} \left\{ \frac{1}{n \Delta_n} \sum \limits_{i=1}^n g_t^{\prime (\alpha)}(\Deli X) \Truniv - N_{\rho^{\prime}_{\alpha}}(t) \right\}.
\end{eqnarray*}
Then, of course, we have $G_{\rho}^{(n)}(t) = G_{\rho,n}^{(\alpha)}(t) + G_{\rho,n}^{\prime (\alpha)}(t)$.

A standard argument laid out in the Appendix shows that it suffices to prove three auxiliary lemmas in order to establish Theorem \ref{ConvThm}. The first one regards the behaviour of the large jumps, i.e.\ it holds for $G_{\rho,n}^{(\alpha)}$ and a fixed $\alpha > 0$. 

\begin{lemma} \label{step1}
If Condition \ref{EasierCond} is satisfied, we have the weak convergence
$$G_{\rho,n}^{(\alpha)} \weak \Gb_{\rho_{\alpha}}$$
in $\linfr$ for each fixed $\alpha >0$, where $\Gb_{\rho_{\alpha}}$ denotes a tight centered Gaussian process with covariance function 
$$H_{\rho_{\alpha}}(u,v) = \int \rho_{\alpha}^2(x) \mathtt 1_{(- \infty, u \wedge v]}(x) \nu (dx).$$
The sample paths of $\Gb_{\rho_{\alpha}}$ are almost surely uniformly continuous with respect to the semimetric
$$d_{\rho_{\alpha}}(u,v) = \left\{ \int \rho_{\alpha}^2(x) \mathtt 1_{(u \wedge v, u \vee v]}(x) \nu(dx) \right\}^{1/2}.$$
\end{lemma}

The general idea behind the proof of Lemma \ref{step1} is to approximate the distribution function $\overline N_{\rho_\alpha}^{(n)}$ with empirical distribution functions of suitable L\'evy processes, for which we can show weak convergence to a Gaussian process using a central limit theorem for empirical processes. Precisely, let $\mu$ be the Poisson random measure associated with the jumps of $X$. Then we consider the L\'evy processes 
$$L^{(n)} = (x \Trunx) \star \mu$$
with the truncation $v_n = \gamma \Delta_n^{\ovw}$ as above. Note that these processes are well-defined, even when the jumps are not summable. The auxiliary empirical processes are defined in terms of a function $f$, for which we plug in $\rho_\alpha$ and $\rho_{\alpha}^{\prime}$ later. Precisely, 
\begin{multline}
\label{EmpPrYDefEqn}
Y_f^{(n)}(t) = \sqrt{n \Delta_n} \bigg \{ \frac{1}{n \Delta_n} \sum \limits_{i=1}^n [f(\Deli L^{(n)}) \Indit (\Deli L^{(n)})  \\
 - \Eb (f(\Deli L^{(n)}) \Indit (\Deli L^{(n)}))] \bigg \}
\end{multline}
for $t \in \R$, where $f \colon \R \rightarrow \R$ is a continuous function which satisfies $|f(x)| \leq K(1 \wedge x^2)$ for some $K>0$. Since $f$ is bounded, expectations always exist. 

Proving weak convergence of the empirical processes $Y_f^{(n)}$ is advantageous, as they consist of a sum of independent increments for which standard tools are available. We begin, however, with a claim which is needed to control the estimation error, as it proves that the bias due to estimating $\Eb (f(\Deli L^{(n)}) \Indit (\Deli L^{(n)}))$ instead of $N_f(t)$ is small compared to the rate of convergence. Due to the simple structure of the L\'evy processes $L^{(n)}$ the proof holds under much weaker conditions than in \cite{Nic15} in their Proposition 17. 

\begin{proposition}
\label{BiasAbschProp}
Suppose Condition \ref{EasierCond} is satisfied and let $f \colon \R \rightarrow \R$ be a Borel-measurable function with $|f(x)| = O(|x|^p)$ as 
$|x| \rightarrow 0$ and $|f(x)| \leq K(1 \wedge x^2)$ for all $x \in \R$ and a $K>0$. Then we have
\begin{align}
\label{BiasAbschEqn}
\sup \limits_{t \in \overline \R} \left| \frac{1}{\Delta_n} \Eb \left\{ f(L^{(n)}_{\Delta_n}) \Indit (L^{(n)}_{\Delta_n}) \right\} - N_f(t) \right| = O(\Delta_n^{\ovw}),
\end{align}
with $\overline \R = \R \cup \lbrace - \infty, + \infty \rbrace$.
\end{proposition}

The following claim now states weak convergence of $Y_f^{(n)}$. Its proof relies heavily on a result from \cite{Kos08} which is tailored for triangular arrays of independent processes. 

\begin{proposition}
\label{LevyCLTProp}
Suppose Condition \ref{EasierCond} is satisfied and let $f \colon \R \rightarrow \R$ be a continuous function with $|f(x)| \leq K(1 \wedge |x|^p)$ for all $x \in \R$ and some $K>0$. Then the empirical processes $Y_f^{(n)}$ from \eqref{EmpPrYDefEqn} converge weakly in $\linfr$ to the tight mean zero Gaussian process $\Gb_f$ from Lemma \ref{step1}, that is
$$Y^{(n)}_f \weak \Gb_f.$$
\end{proposition}

Using the previous two propositions, the final part of the proof of Lemma \ref{step1} is the justification that the error is small when replacing the original increments by those of the approximating L\'evy processes. This argument is laid out in the Appendix as well.

%\begin{remark}
%\label{RescBMRem}
%For $f$ as in Proposition \ref{LevyCLTProp} let $c_f = \int f^2(x) \nu(dx)$ and consider a standard Brownian motion $\Bb$ on 
%$[0,c_f]$. It is a well known fact (see for instance Section 8 in \cite{Bil99}) that the law of this process is tight in $\ell^{\infty}([0,c_f])$. Furthermore, the sample paths of $\Bb$ are uniformly continuous with respect to the Euclidean distance and for each $\epsilon, \eta >0$ there is a $\delta >0$ such that
%\begin{eqnarray}
%\label{EquicontAbsch}
%\Prob \big ( \sup \limits_{|u-v|^{1/2} < \delta} \left| \Bb (u) - \Bb(v) \right| > \epsilon \big ) < \eta.
%\end{eqnarray}
%This is a consequence of Theorem 1.5.7 and Addendum 1.5.8 in \cite{VanWel96}. 
%
%Because of Lemma 1.3.12(ii) in the previously mentioned reference two tight Borel laws on $\ell^{\infty}(T)$ (for an arbitrary set $T$) are equal if they have the same marginal distributions. Therefore the limit distribution of Proposition \ref{LevyCLTProp} is equal to the law of the rescaled Brownian motion
%$$\Bb_f(t) = \Bb \bigg ( \int f^2(x) \mathtt 1_{(- \infty, t]}(x) \nu(dx) \bigg ),$$
%because the latter process is in fact tight in $\linfr$ by \eqref{EquicontAbsch} and Theorem 1.5.6 in \cite{VanWel96}. The sample paths of $\Bb_f$ are also uniformly continuous with respect to $d_f$. \qed
%\end{remark}

In order to obtain the result from Theorem \ref{ConvThm} we have to ensure that the limiting process $\Gb_{\rho_{\alpha}}$ converges in a suitable sense as $\alpha \rightarrow 0$. This is the content of the second lemma.   

\begin{lemma} \label{step2}
Under Condition \ref{EasierCond} the weak convergence
$$\Gb_{\rho_{\alpha}} \weak \Gb_{\rho}$$
holds in $\linfr$ as $\alpha \rightarrow 0$.
\end{lemma}

Its proof is a direct consequence of the following result.

\begin{proposition}
\label{GalKonvLem}
Suppose Condition \ref{EasierCond} is satisfied and let $f_n \colon \R \rightarrow \R$ ($n \in \N_0$) be Borel-measurable functions with 
$|f_n(x)| \leq K(1 \wedge x^2)$ for a constant $K>0$ and all $n \in \N_0$, $x \in \R$. Assume further that $f_n \rightarrow f_0$ converges $\nu$-a.e. Then we have weak convergence
$$\Gb_{f_n} \weak \Gb_{f_0}$$
in $\linfr$ for $n \rightarrow \infty$.
\end{proposition}

Finally, the contribution due to small jumps, which are comprised in the process $G_{\rho,n}^{\prime (\alpha)}$, need to be uniformly small when $\alpha$ tends to zero. This is discussed in the next lemma. 

\begin{lemma} \label{step3}
Suppose Condition \ref{EasierCond} is satisfied. Then for each $\eta >0$ we have:
$$\lim \limits_{\alpha \rightarrow 0} \limsup \limits_{n \rightarrow \infty} \Prob (\sup \limits_{t \in \R} 
|G_{\rho,n}^{\prime (\alpha)}(t)| > \eta) = 0.$$
\end{lemma}

\section{Appendix} \label{sec:app}
\def\theequation{4.\arabic{equation}}
\setcounter{equation}{0}

Before we prove Theorem \ref{ConvThm} and the other claims related to it, we begin with a set of alternative conditions. Here and below, $K$ or $K(\delta)$ denote generic constants which sometimes depend on an auxiliary quantity $\delta$ and may change from line to line.

\begin{condition}
\label{Cond1} 
\begin{compactenum}[(a)]
\item \textit{Conditions on the L\'evy measure and the function $\rho$:} 
			\begin{enumerate}[(1)]
			\item There exists $r \in [0,2]$ with $\int \big ( 1 \wedge |x|^{r+\delta} \big ) \nu(dx) < \infty$ for each $\delta >0$.
			      \label{BlGetCond}
			\item \label{RhoCond}
            $\rho \colon \R \rightarrow \R$ is a bounded $\mathcal C^1$-function with $\rho(0)=0$. Furthermore, there exists some $p > 2 \vee (1+ 3r)$
			      such that the derivative satisfies $| \rho^{\prime}(x) | \leq K |x|^{p-1}$ for all $x \in \R$ and some $K>0$. 
			\item $\int |x|^{p-1} \mathtt 1_{\lbrace |x| \geq 1 \rbrace} \nu(dx) < \infty$ with $p$ from \eqref{RhoCond}.
						\label{LevMeaMomCond}
			\item \begin{compactenum}[(I)]
						\item There exist $\ovr > \ovw > \ovv >0$, $\alpha_0 >0$, $q >0$ and $K>0$ such that we have for sufficiently large 
						      $n \in \mathbb N$: \label{FiLevyDistCond}
						      $$\int \int \mathtt 1_{\lbrace |u-z| \leq \Delta_n^{\ovr} \rbrace} \mathtt 1_{\lbrace \Delta_n^{\ovv} /2 < |u| \leq 
									\alpha_0 \rbrace } \mathtt 1_{\lbrace \Delta_n^{\ovv} /2 < |z| \leq \alpha_0 \rbrace} \nu(dz) \nu(du) \leq K \Delta_n^q.$$
						\item For each $\alpha >0$ there is a $K(\alpha)>0$, with
						      $$\int \int \mathtt 1_{\lbrace |u-z| \leq \Delta_n^{\ovr} \rbrace} \mathtt 1_{\lbrace |u| > \alpha \rbrace} 
									\mathtt 1_{\lbrace |z| > \alpha \rbrace} \nu(dz) \nu(du) \leq K(\alpha) \Delta_n^q,$$
									for $n \in \N$ large enough with the constants from \eqref{FiLevyDistCond}.
									\label{SeLevyDistCond}
						\end{compactenum}
			\end{enumerate}
\item \textit{Conditions on the truncation sequence $v_n$ and the observation scheme:} \\
      \label{ObsSchCondit}
      We have $v_n = \gamma \Delta_n^{\ovw}$ for some $\gamma >0$ and $\ovw$ satisfying
			$$\frac{1}{2(p-r)} < \ovw < \frac{1}{2} \wedge \frac{1}{4r}.$$
      Furthermore, the observation scheme satisfies with the constants from the previous conditions: 
      \begin{enumerate}[(1)]
			\item $\Delta_n \rightarrow 0,$
			     
			\item $n \Delta_n \rightarrow \infty,$

			\item $n \Delta_n^{1+ q/2} \rightarrow 0,$
			      \label{ObsSchCond3}
			\item $n \Delta_n^{1+2 \ovw} \rightarrow 0,$
			      \label{ObsSchCond4}
			\item $n \Delta_n^{2p \ovv -1} \rightarrow 0,$
			      \label{ObsSchCond5}
			\item $n \Delta_n^{2(1-r \ovw (1+ \epsilon))} \rightarrow 0$ for some $\epsilon >0$,
			      \label{ObsSchCond6}
			\item $n \Delta_n^{1+ 2(\ovr - \ovw)} \rightarrow \infty.$
			      \label{ObsSchCond7}
			\end{enumerate}
\item \label{DrianDiffCond} \textit{Conditions on the drift and the diffusion coefficient:} \\
			 Set
			\begin{align}
			\label{Tildeelldef}
			\tilde \ell = \begin{cases}
			                     \frac{1}{2} \big ( \epsilon \wedge \frac{1-2r \ovw}{2r \ovw} \big ), \quad \text{ if } r \leq 1  \\
													 \frac{1}{2} \big ( \epsilon \wedge \frac{1-2r \ovw}{2r \ovw} \wedge \frac{2(p-r)\ovw - 1}{2(r-1) \ovw} \big ),
													 \quad \text{ if } r >1
													 \end{cases}
			\end{align}
			with the previously established constants and $\ell = 1+ \tilde \ell$. Let
			$$m = \left \lfloor \frac{2+ r \ell}{\ell -1} \vee \frac{1+ 2 \ovw}{1/2 - \ovw} \right \rfloor + 1.$$
			There is a random variable $A$ such that
			$$|b_s(\omega)| \leq A(\omega), |\sigma_s(\omega)| \leq A(\omega) \text{ for all } (\omega,s) \in \Omega \times \R_+$$
			and
			$$\Eb A^m < \infty.$$
\end{compactenum}
\end{condition}

In the following, we will work with the previous assumptions without further mention. This is due to the following result which proves that Condition \ref{EasierCond} implies the set of conditions above. 

\begin{proposition} \label{easier}
Condition \ref{EasierCond} is sufficient for Condition \ref{Cond1}.
\end{proposition}

\begin{proof}
Let $0< \beta < 2$, $0< \zeta < \tau < 1/16$, $p = 8(1+3 \beta) \frac{1+ \tau}{1-16 \tau}$ and suppose that Condition \ref{EasierCond} is satisfied for these constants. In order to verify Condition \ref{Cond1} define the following quantities:
\begin{align}
\label{KonstDefEq}
r &= \beta, \quad \ovr = \frac{1+ \zeta}{8}, \quad \ovw = 1/8, \nonumber \\
\ovv &= \frac{1-16 \tau}{8(1+3 \beta)}, \quad q = \ovr - (1+ 3 \beta) \ovv = \zeta/8 + 2 \tau.
\end{align}
$\rho$ is obviously suitable for Condition \ref{Cond1}\eqref{RhoCond}, and in particular $p > 2 \vee (1+3r)$ is clearly satisfied. Condition \ref{Cond1}\eqref{ObsSchCondit} is established since 
$$\frac{1}{2(p-r)} < \ovw < \frac{1}{2} \wedge \frac{1}{4r}$$
holds due to $p > 4 + \beta$, and further simple calculations show
\begin{align*}
1 < 1 + 2 \ovr - 2 \ovw &< t_2^{-1} = 1+ \zeta \\ 
            &< 1+ \tau = t_1^{-1} < (2 p \ovv - 1) \wedge (1+ \frac{q}{2}) \wedge (1+2 \ovw) < 2- 2 r \ovw (1 + \epsilon),
\end{align*}
with $\epsilon = \frac{3- \beta}{2}$. Therefore, all conditions on the observation scheme are satisfied. 

Additionally, we have
$$h(x) (1 \wedge |x|^{r+\delta}) \leq K|x|^{-(1- \delta)}$$
on a neighbourhood of zero for each $\delta >0$. Therefore and due to Condition \ref{EasierCond}\eqref{DiCondmiddle} and \eqref{DiCondinfty} we have $\int \big ( 1 \wedge |x|^{r+ \delta} \big ) \nu(dx) < \infty$ for every $\delta >0$. Again conditions \ref{EasierCond}\eqref{DiCondmiddle} and \eqref{DiCondinfty} prove $\int |x|^{p-1} \mathtt 1_{\lbrace |x| \geq 1 \rbrace} \nu(dx) < \infty$ which is Condition \ref{Cond1}\eqref{LevMeaMomCond}. 

With the constants given above we obtain for $\tilde \ell$ defined in \eqref{Tildeelldef}
\begin{align*}
\tilde \ell = \epsilon /2 = \frac{3- \beta}{4}
\end{align*}
and therefore with $\ell = 1+ \tilde \ell$
$$\frac{1+ 2 \ovw}{\frac{1}{2} - \ovw} = 10/3 < 4 \quad \text{ and } \quad \frac{2+ r \ell}{\ell -1} = \frac{8 +7 \beta - \beta^2}{3 - \beta}  \in (\frac{8}{3},18).$$
Thus Condition \ref{EasierCond}\eqref{DriDiffMomCond} yields Condition \ref{Cond1}\eqref{DrianDiffCond}. 

We are thus left with proving Condition \ref{Cond1}\eqref{FiLevyDistCond} and \eqref{SeLevyDistCond}. Obviously, $0 < \ovv < \ovw < \ovr$ holds with the choice in \eqref{KonstDefEq}. First we verify Condition \ref{Cond1}\eqref{FiLevyDistCond}. To this end we choose $\alpha_0 >0$ such that $h(x) \leq K |x|^{-(1+\beta)}$ on $[- \alpha_0, \alpha_0] \setminus \lbrace 0 \rbrace$. Now we compute for $n \in \N$ large enough:
\begin{align*}
\int &\int \mathtt 1_{\lbrace |u-z| \leq \Delta_n^{\ovr} \rbrace} \mathtt 1_{\lbrace \Delta_n^{\ovv} /2 < |u| \leq \alpha_0 \rbrace } 
                  \mathtt 1_{\lbrace \Delta_n^{\ovv} /2 < |z| \leq \alpha_0 \rbrace} \nu(dz) \nu(du)  \\
  &\leq K \int \int \mathtt 1_{\lbrace |u-z| \leq \Delta_n^{\ovr} \rbrace} \mathtt 1_{\lbrace \Delta_n^{\ovv} /2 < |u| \leq \alpha_0 \rbrace } 
                  \mathtt 1_{\lbrace \Delta_n^{\ovv} /2 < |z| \leq \alpha_0 \rbrace} |z|^{-(1+\beta)} |u|^{-(1+ \beta)} dz du  \\
&\leq 2 K\int \limits_0^{\infty} \int \limits_0^{\infty} \mathtt 1_{\lbrace |u-z| \leq \Delta_n^{\ovr} \rbrace} 
                \mathtt 1_{\lbrace \Delta_n^{\ovv}/2 < u \leq \alpha_0 \rbrace } \mathtt 1_{\lbrace \Delta_n^{\ovv}/2 < z \leq \alpha_0 
								\rbrace} z^{-(1+\beta)} u^{-(1+\beta)} dz du. 
\end{align*}
For the second inequality we have used symmetry of the integrand as well as $\Delta_n^{\ovr} < \Delta_n^{\ovv}/2$. In the following, we ignore the extra condition on $u$. Evaluation of the integral with respect to $u$ plus a Taylor expansion give the further upper bounds
\begin{align}
\label{FracIntEq}
  &K \int \limits_0^{\infty} \frac{|(z- \Delta_n^{\ovr})^{\beta} - (z+ \Delta_n^{\ovr})^{\beta}|}
                     {|z^2 - \Delta_n^{2 \ovr} |^{\beta}} 
										z^{-(1+ \beta)} \mathtt 1_{\lbrace \Delta_n^{\ovv}/2 < z \leq \alpha_0 \rbrace} dz \nonumber \\
    &\leq K \Delta_n^{\ovr} \int \limits_0^{\infty} \frac{\xi(z)^{\beta-1}}{|z^2 - \Delta_n^{2 \ovr}|^{\beta}} 
                      z^{-(1+ \beta)} \mathtt 1_{\lbrace \Delta_n^{\ovv}/2 < z \leq \alpha_0 \rbrace} dz
\end{align}
for some $\xi(z) \in [z- \Delta_n^{\ovr}, z + \Delta_n^{\ovr}]$. Finally, we distinguish the cases $\beta < 1$ and $\beta \geq 1$ for which the numerator has to be treated differently, depending on whether it is bounded or not. The denominator is always smallest if we plug in $\Delta_n^{\ovv}/2$ for $z$. Overall,
\begin{align*}
\int &\int \mathtt 1_{\lbrace |u-z| \leq \Delta_n^{\ovr} \rbrace} \mathtt 1_{\lbrace \Delta_n^{\ovv} /2 < |u| \leq \alpha_0 \rbrace } 
                  \mathtt 1_{\lbrace \Delta_n^{\ovv} /2 < |z| \leq \alpha_0 \rbrace} \nu(dz) \nu(du) \nonumber \\
     &\leq \begin{cases}
                       K \Delta_n^{\ovr} \Delta_n^{-(1+\beta) \ovv} \int \limits_{\Delta_n^{\ovv}/2}^{\alpha_0} 
											 z^{-(1+ \beta)} dz, \quad \text{ if } \beta <1 \\
											 K \Delta_n^{\ovr} \Delta_n^{-2\beta \ovv} \int \limits_{\Delta_n^{\ovv}/2}^{\alpha_0} 
											 z^{-(1+ \beta)} dz, \quad \text{ if } \beta \geq 1
											\end{cases} \\
&\leq K \Delta_n^{\ovr - (1+ 3 \beta)\ovv} = K \Delta_n^q \nonumber.
\end{align*}

Finally, we consider Condition \ref{Cond1}\eqref{SeLevyDistCond}, for which we proceed similarly with $n \in \N$ large enough and $\alpha >0$ arbitrary:
\begin{eqnarray*}
\lefteqn{\int \int \mathtt 1_{\lbrace |u-z| \leq \Delta_n^{\ovr} \rbrace} \mathtt 1_{\lbrace |u| > \alpha \rbrace } 
                  \mathtt 1_{\lbrace  |z| > \alpha \rbrace} \nu(dz) \nu(du) }& & \nonumber \\
&\leq& O(\Delta_n^{\ovr}) + 2 K \int \limits_{M^{\prime}}^{\infty} \int \limits_{M^{\prime}}^{\infty} 
                  \mathtt 1_{\lbrace |u-z| \leq \Delta_n^{\ovr} 
                  \rbrace} \mathtt 1_{\lbrace u > \alpha \rbrace } \mathtt 1_{\lbrace z > \alpha \rbrace} 
									z^{-4} u^{-4} dz du.
\end{eqnarray*}
This inequality holds with a suitable $M^{\prime} >0$ due to Condition \ref{EasierCond} \eqref{DiCondmiddle} and \eqref{DiCondinfty}, as we have $h(x) \leq K|x|^{-4}$ for large $|x|$ from $p>4$. Therefore,
\begin{align}
\label{IntalphaAbsch}
\int &\int \mathtt 1_{\lbrace |u-z| \leq \Delta_n^{\ovr} \rbrace} \mathtt 1_{\lbrace |u| > \alpha \rbrace } 
                  \mathtt 1_{\lbrace  |z| > \alpha \rbrace} \nu(dz) \nu(du)  \nonumber \\
	   &\leq O(\Delta_n^{\ovr}) + K \int \limits_{M^{\prime}}^{\infty} \frac{|(z- \Delta_n^{\ovr})^{3} - 
       (z+ \Delta_n^{\ovr})^{3}|} {|z^2 - \Delta_n^{2 \ovr} |^{3}} z^{-4} 
			 \mathtt 1_{\lbrace z > \alpha \rbrace} dz \nonumber \\
     &\leq O(\Delta_n^{\ovr}) + K \Delta_n^{\ovr} \int \limits_{M^{\prime} \vee \alpha}^{\infty} 
                        \frac{(\xi(z))^2} {|z^2 - \Delta_n^{2 \ovr} |^{3}} z^{-4} dz = o(\Delta_n^q),
\end{align}
using another Taylor expansion as in \eqref{FracIntEq} with a $\xi(z) \in [z- \Delta_n^{\ovr} , z + \Delta_n^{\ovr}]$. The final bound in \eqref{IntalphaAbsch} holds since the last integral is finite.
\end{proof}

Let us now proceed with a proof of the results from Section \ref{sec:converge}. We begin with the results in order to establish Lemma \ref{step1}. \\

\textbf{Proof of Proposition \ref{BiasAbschProp}.}
With the notation $\overline F_n = \lbrace x \colon |x| > v_n \rbrace$ we have $L^{(n)} = x \mathtt 1_{\overline F_n} (x) \star \mu$. These processes are compound Poisson processes and possess the representation
$$L^{(n)}_t = \sum \limits_{i=1}^{\overline N^n_t} Y_i^{(n)},$$
where $\overline N^n$ is a Poisson process with parameter $\nu(\overline F_n)$ and $(Y_i^{(n)})_{i \in \mathbb N}$ is an i.i.d. sequence of random variables with distribution $1/\nu(\overline F_n) \times \nu \vert_{\overline F_n}$ which is independent of $\overline N^n$.

Now consider the sets $A_n = \left \{ \overline N^n_{\Delta_n} \leq 1 \right \}$. According to Condition \ref{Cond1}\eqref{BlGetCond} we have $\int \big ( 1 \wedge |x|^{r+\delta} \big ) \nu(dx) < \infty$ for each $\delta >0$. Thus there is a constant $K(\delta) >0$ such that $\nu(\overline F_n) \leq K(\delta) v_n^{-r-\delta}$. Consequently, 
\begin{eqnarray}
\label{JumProbAbsch}
\Prob(A_n^C) \leq \Delta_n^2 (\nu(\overline F_n))^2 \leq K(\delta) \Delta_n^{2-2(r+ \delta) \ovw} = O(\Delta_n^{3/2 - \epsilon}),
\end{eqnarray}
where the final equality holds for each $\epsilon >0$ as soon as $\delta > 0$ is small enough due to $\ovw < \frac{1}{4r}$. Now we obtain
\begin{align*}
\frac{1}{\Delta_n} &\Eb \left\{ f(L^{(n)}_{\Delta_n}) \Indit (L^{(n)}_{\Delta_n}) \right\} - N_f(t) \nonumber \\
&= \frac{1}{\Delta_n} \left\{ \int \limits_{A_n^C} f(L^{(n)}_{\Delta_n}) \Indit (L^{(n)}_{\Delta_n}) \Prob(d \omega) +
\int \limits_{A_n} f(L^{(n)}_{\Delta_n}) \Indit (L^{(n)}_{\Delta_n}) \Prob(d \omega) \right\} - \nonumber \\ 
&- \int f(x) \Indit(x) \nu(dx)  \nonumber \\
&= O(\Delta_n^{\ovw}) + \frac{1}{\Delta_n} \int \limits_{\lbrace \overline N^n_{\Delta_n} = 1 \rbrace} f(Y^{(n)}_1) \Indit(Y^{(n)}_1) \Prob(d \omega) \nonumber \\
&- \int f(x) \Indit(x) \nu(dx) 
\end{align*}
where the $O$-term is uniform in $t \in \overline \R$. In the final equality above we have bounded the first term in the curly brackets using \eqref{JumProbAbsch} as well as $\ovw < 1/2$. Using the properties of a compound Poisson process we obtain 
\begin{align*}
\frac{1}{\Delta_n} &\Eb \left\{ f(L^{(n)}_{\Delta_n}) \Indit (L^{(n)}_{\Delta_n}) \right\} - N_f(t) \\
&= e^{- \Delta_n \nu( \overline F_n)} \int \limits_{\overline F_n} f(x) \Indit(x) \nu(dx) - \int f(x) \Indit(x) \nu(dx) + O(\Delta_n^{\ovw})
\end{align*}
with a uniform $O$-term. Since we have $|1 - e^{-x}| \leq |x|$ for $x \geq 0$ and as $f$ is integrable with respect to $\nu$, we conclude
\begin{align}
\label{LastOAbschEqn}
\sup \limits_{t \in \overline \R} &\left| \frac{1}{\Delta_n} \Eb \left\{ f(L^{(n)}_{\Delta_n}) \Indit (L^{(n)}_{\Delta_n}) \right\} - N_f(t) \right| \nonumber \\
&\leq O(\Delta_n^{\ovw}) + K \Delta_n \nu(\overline F_n) + \int |f(x)| \mathtt 1_{\lbrace |x| \leq v_n \rbrace} \nu(dx) \nonumber \\
&\leq O(\Delta_n^{\ovw}) + K \Delta_n^{1/2} + K v_n^{p-r-\delta} \int \big ( 1 \wedge |x|^{r+ \delta} \big ) \nu(dx) = O(\Delta_n^{\ovw}),
\end{align}
where the last inequality holds for two reasons: First, for each $\delta >0$ and $n \in \N$ large enough we use $|f(x)| \leq K |x|^p$ on $\lbrace |x| \leq v_n \rbrace$, which is possible due to $f(x) = O(|x|^p)$ as $|x| \rightarrow 0$. Second, $1 - (r + \delta) \ovw \geq 3/4 - \epsilon \geq 1/2 > \ovw$ from $\ovw < \frac{1}{4r}$. For the final equality in \eqref{LastOAbschEqn} observe that $p-r- \delta >1$ for $\delta > 0$ small enough and $v_n = \gamma \Delta_n^{\ovw}$ as well as Condition \ref{Cond1}\eqref{BlGetCond}.
\qed

\medskip

\textbf{Proof of Proposition \ref{LevyCLTProp}.} The processes $Y^{(n)}_f$ have the form
$$Y^{(n)}_f (\omega;t) = \sum \limits_{i=1}^{m_n} \left\{ g_{ni}(\omega;t) - \Eb (g_{ni}(\cdot;t)) \right\},$$
with $m_n = n$ and the triangular array $\lbrace g_{ni}(\omega;t) \mid n \in \N; i = 1, \ldots, n ; t \in \R \rbrace$ of processes
$$g_{ni}(\omega ; t) = \frac{1}{\sqrt{n \Delta_n}} f(\Deli L^{(n)}(\omega)) \Indit( \Deli L^{(n)}(\omega)),$$
which is obviously independent within rows. Thus by Theorem 11.16 in \cite{Kos08} the proof is complete, if we can show the following six conditions of the triangular array $\lbrace g_{ni} \rbrace$ (see for instance \cite{Kos08} for the notions of AMS and manageability):
\begin{enumerate}[(A)]
\item $\lbrace g_{ni} \rbrace$ is almost measurable Suslin (AMS);
      \label{Liste1}
\item $\lbrace g_{ni} \rbrace$ is manageable with envelopes $\lbrace G_{ni} \mid n \in \N; i = 1, \ldots,n \rbrace$, given through
      $G_{ni} = \frac{K}{\sqrt{n \Delta_n}} (1 \wedge \left| \Deli L^{(n)} \right|^{p})$ with $K >0$ such that 
			$|f(x)| \leq K(1 \wedge |x|^{p})$, where $\lbrace G_{ni} \rbrace$ are also independent within rows;
			\label{Liste2}
\item $$H_f(u,v) = \lim \limits_{n \rightarrow \infty} \Eb \left\{ Y^{(n)}_f(u) Y^{(n)}_f(v) \right\}$$ for all $u,v \in \R$; \label{Liste3} 
\item $\limsup \limits_{n \rightarrow \infty} \sum \limits_{i=1}^n \Eb G_{ni}^2 < \infty$; \label{Liste4}
\item $$\lim \limits_{n \rightarrow \infty} \sum \limits_{i=1}^n \Eb G_{ni}^2 \mathtt 1_{\lbrace G_{ni} > \epsilon \rbrace} =0$$ for each $\epsilon >0$;
			\label{Liste5}
\item For $u,v \in \R$ the limit $d_f(u,v) = \lim \limits_{n \rightarrow \infty} d_f^{(n)}(u,v)$ with
      $$d_f^{(n)}(u,v) = \left\{ \sum \limits_{i=1}^n \Eb \left| g_{ni}(\cdot; u) - g_{ni}( \cdot;v) \right|^2 \right\}^{1/2}$$
			exists, and for all deterministic sequences $(u_n)_{n \in \N},(v_n)_{n \in \N} \subset \R$ with $d_f(u_n,v_n) \rightarrow 0$ we also have $d_f^{(n)}(u_n,v_n) \rightarrow 0$.
			\label{Liste6}
\end{enumerate}

\smallskip

\textit{Proof of  \eqref{Liste1}.} With Lemma 11.15 in \cite{Kos08} the triangular array $\lbrace g_{ni} \rbrace$ is AMS if it is separable, that is for each 
$n \in \N$ there exists a countable subset $S_n \subset \R$ such that
$$\Prob^{\ast} \bigg( \sup \limits_{t_1 \in \R} \inf \limits_{t_2 \in S_n} \sum \limits_{i=1}^n (g_{ni}(\omega; t_1) - 
g_{ni}(\omega;t_2))^2 >0 \bigg) = 0.$$
But if we choose $S_n = \Q$ for all $n \in \N$, we obtain
$$\sup \limits_{t_1 \in \R} \inf \limits_{t_2 \in S_n} \sum \limits_{i=1}^n (g_{ni}(\omega; t_1) - 
g_{ni}(\omega;t_2))^2 = 0$$
for each $\omega \in \Omega$ and $n \in \N$.

\smallskip

\textit{Proof of  \eqref{Liste2}.}
$G_{ni}$ are independent within rows since $L^{(n)}$ are L\'evy processes. 

In order to show manageability consider for $n \in \N$ and $\omega \in \Omega$ the set
\begin{multline*}
\mathcal G_{n \omega} = \bigg\{ \bigg( \frac{1}{\sqrt{n \Delta_n}} f(\Delta_1^n L^{(n)}(\omega)) \Indit(\Delta_1^n L^{(n)}(\omega)), \ldots \\ 
\ldots, \frac{1}{\sqrt{n \Delta_n}} f(\Delta_n^n L^{(n)}(\omega)) \Indit(\Delta_n^n L^{(n)}(\omega)) \bigg) \bigg{|} t \in \R \bigg\} \subset \R^n.
\end{multline*}
These sets are bounded with envelope vector
$$G_n(\omega) = (G_{n1}(\omega) , \ldots, G_{nn}(\omega)) \in \R^n.$$
For $i_1, i_2 \in \lbrace 1, \ldots, n \rbrace$ the projection 
\begin{multline*}
p_{i_1, i_2}(\mathcal G_{n \omega}) = \bigg \{ \bigg( \frac{1}{\sqrt{n \Delta_n}} f(\Delta_{i_1}^n L^{(n)}(\omega)) \Indit(\Delta_{i_1}^n L^{(n)}(\omega)), \\
 \frac{1}{\sqrt{n \Delta_n}} f(\Delta_{i_2}^n L^{(n)}(\omega)) \Indit(\Delta_{i_2}^n L^{(n)}(\omega)) \bigg) \mid t \in \R \bigg \} \subset \R^2
\end{multline*}
onto the $i_1$-th and the $i_2$-th coordinate is an element of the set
\begin{align*}
\bigg\{ \lbrace (0,0) \rbrace, &\lbrace (0,0), (s_{i_1,n}(\omega),0) \rbrace, \lbrace (0,0), (0,s_{i_2,n}(\omega)) \rbrace, 
\lbrace (0,0), (s_{i_1,n}(\omega), s_{i_2,n}(\omega)) \rbrace, \\
\lbrace (0,0), &(s_{i_1,n}(\omega), 0), (s_{i_1,n}(\omega), s_{i_2,n}(\omega)) \rbrace,
\lbrace (0,0), (0, s_{i_2,n}(\omega)), (s_{i_1,n}(\omega), s_{i_2,n}(\omega)) \rbrace \bigg\}.
\end{align*}
with $s_{i_1,n}(\omega) = \frac{1}{\sqrt{n \Delta_n}} f(\Delta_{i_1}^n L^{(n)}(\omega))$ and 
$s_{i_2,n}(\omega) = \frac{1}{\sqrt{n \Delta_n}} f(\Delta_{i_2}^n L^{(n)}(\omega))$. 

Consequently, in the sense of Definition 4.2 in \cite{Pol90}, for every $s \in \R^2$ no proper coordinate projection of $\mathcal G_{n \omega}$ can surround $s$ and therefore $\mathcal G_{n \omega}$ has a pseudo dimension of at most $1$ (Definition 4.3 in \cite{Pol90}). Thus by Corollary 4.10 in the same reference, there exist constants $A$ and $W$ which depend only on the pseudo dimension such that
$$D_2(x \| \alpha \odot G_n(\omega) \|_2, \alpha \odot \mathcal G_{n \omega}) \leq A x^{-W} =: \lambda(x),$$
for all $0< x \leq 1$, $n \in \N$, $\omega \in \Omega$ and each rescaling vector $\alpha \in \R^n$ with non-negative entries, where 
$\| \cdot \|_2$ denotes the Euclidean distance on $\R^n$, $D_2$ denotes the packing number with respect to the Euclidean distance and 
$\odot$ denotes coordinate-wise multiplication. 

Obviously, we have 
$$\int \limits_0^1 \sqrt{\log \lambda(x)} dx < \infty,$$ 
and therefore the triangular array $\lbrace g_{ni} \rbrace$ is indeed manageable with envelopes $\lbrace G_{ni} \rbrace$.
 
\smallskip

\textit{Proof of  \eqref{Liste3}.} Using the independence within rows of the triangular array $\lbrace g_{ni} \rbrace$ we calculate for $u,v \in \R$ as follows:
\begin{align*}
\Eb &\left\{ Y^{(n)}_f(u) Y^{(n)}_f(v) \right\} = \frac{1}{n \Delta_n} \sum \limits_{i=1}^n \Eb \big [ f^2(\Deli L^{(n)}) 
                \mathtt 1_{(- \infty, u \wedge v]}( \Deli L^{(n)}) \big ] - \\
					&- \frac{1}{n \Delta_n} \sum \limits_{i=1}^n \bigg ( \Eb \big [ f(\Deli L^{(n)}) 
					     \mathtt 1_{(- \infty, u]}(\Deli L^{(n)}) \big ] \Eb \big [ f(\Deli L^{(n)}) \mathtt 1_{(- \infty, v]}(\Deli L^{(n)}) 
							\big ] \bigg ) \\
					&= \frac{1}{\Delta_n} \Eb \left\{ f^2(L^{(n)}_{\Delta_n}) \mathtt 1_{(- \infty, u \wedge v]}(L^{(n)}_{\Delta_n}) 
					   \right\}- \\
					&- \Delta_n \bigg ( \frac{1}{\Delta_n} \Eb \left\{ f(L^{(n)}_{\Delta_n}) \mathtt 1_{(- \infty, u ]}(L^{(n)}_{\Delta_n}) \right\} 
					   \bigg )
					   \bigg ( \frac{1}{\Delta_n} \Eb \left\{ f(L^{(n)}_{\Delta_n}) \mathtt 1_{(- \infty, v]}(L^{(n)}_{\Delta_n}) \right\} \bigg) \\
					&\rightarrow \int f^2(x) \mathtt 1_{(- \infty, u \wedge v]}(x) \nu(dx) = H_f(u,v).
\end{align*}
The equality holds because $\lbrace g_{ni}(t) \rbrace$ are also identically distributed within rows and the convergence follows with Proposition \ref{BiasAbschProp}.

\smallskip

\textit{Proof of  \eqref{Liste4}.}
 Because $L^{(n)}$ are L\'evy processes we obtain
\begin{align*}
\limsup \limits_{n \rightarrow \infty} \sum \limits_{i=1}^n \Eb G_{ni}^2 &= 
\limsup \limits_{n \rightarrow \infty} K^2 \frac{1}{\Delta_n} \Eb \left\{ 1 \wedge \left| L^{(n)}_{\Delta_n} \right|^{2 p} \right\} \\
&= K^2 \int (1 \wedge |x|^{2 p}) \nu(dx) < \infty,
\end{align*}
with Proposition \ref{BiasAbschProp}, since $p>1$.

\smallskip
\textit{Proof of  \eqref{Liste5}.} 
 We have $n \Delta_n \rightarrow \infty$ and thus for $\epsilon >0$ we can choose 
$$N_{\epsilon} = \min \lbrace m \in \N \mid \frac{K}{\sqrt{n \Delta_n}} \leq \epsilon \text{ for all } n \geq m \rbrace < \infty.$$
So for $n \geq N_{\epsilon}$ the integrand satisfies $G_{ni}^2 \mathtt 1_{\lbrace G_{ni} > \epsilon \rbrace} =0$ for all $1 \leq i \leq n$ and this yields the assertion.

\smallskip
\textit{Proof of  \eqref{Liste6}.}
From Proposition \ref{BiasAbschProp} and since the $L^{(n)}$ are L\'evy processes we have
\begin{align*}
\big (d_f^{(n)}(u,v) \big)^2 &= \sum \limits_{i=1}^n \Eb \left| g_{ni}(u) - g_{ni}(v) \right|^2 \nonumber \\
                &=  \frac{1}{n \Delta_n} \sum \limits_{i=1}^n \Eb 
								    \big [ f^2(\Deli L^{(n)}) \mathtt 1_{(u \wedge v, u \vee v]}(\Deli L^{(n)}) \big ] \nonumber \\
							  &=  \frac{1}{\Delta_n} \Eb \left\{ f^2(L^{(n)}_{\Delta_n}) \mathtt 1_{(u \wedge v, u \vee v]}(L^{(n)}_{\Delta_n}) 
								    \right\} \nonumber \\ 
								&= \int f^2(x) \mathtt 1_{( u \wedge v , u \vee v]}(x) \nu(dx) + O(\Delta_n^{\ovw}) \nonumber \\
								&= \big (d_f(u,v) \big)^2 + O(\Delta_n^{\ovw}) \rightarrow \big (d_f(u,v) \big)^2
\end{align*}
for arbitrary $u,v \in \R$, where the $O$-term is uniform in $u,v \in \R$. Therefore,
$$\left| d_f^{(n)}(u,v) - d_f(u,v) \right| = O(\Delta_n^{\ovw /2})$$
uniformly as well, because
$$\left| \sqrt a - \sqrt b \right| \leq \sqrt{ |a-b|}$$
holds for arbitrary $a,b \geq 0$. This uniform convergence implies immediately that for deterministic sequences $(u_n)_{n \in \N}, (v_n)_{n \in \N} \subset \R$ with $d_f(u_n, v_n) \rightarrow 0$ we also have $d_f^{(n)}(u_n, v_n) \rightarrow 0$. 

Finally, $d_f$ is in fact a semimetric: Define for $y \in \R$ the random vectors $g_n(y) = (g_{n1}(y), \ldots, g_{nn}(y)) \in \R^n$ and apply first the triangle inequality in $\R^n$ and afterwards the Minkowski inequality to obtain
\begin{align*}
d_f^{(n)}(u,v) &= \left\{ \Eb \| g_n(u) - g_n(v) \|^2 \right \}^{1/2} \\
               &\leq \left \{ \Eb \big ( \| g_n(u) - g_n(z) \| + \| g_n(z) - g_n(v) \| \big )^2 \right\}^{1/2} \\
							 &\leq \left \{ \Eb \|g_n(u) - g_n(z) \|^2 \right \}^{1/2} + \left \{ \Eb \|g_n(z) - g_n(v) \|^2 \right \}^{1/2} \\
							 &= d_f^{(n)}(u,z) + d_f^{(n)}(z,v),
\end{align*}
for $u,v,z \in \R$ and $n \in \N$. The triangle inequality for $d_f$ follows immediately. 
%applying first the triangle inequality in $\R^n$ and afterwards the Minkowski inequality yields the triangle inequality for $d_f^{(n)}$ and thus with equation \eqref{PntwMetKonEq} that for $d_f$. 
\qed

\medskip 

\textbf{Proof of Lemma \ref{step1}.} Let $\alpha > 0$ be fixed and recall the definition of the processes $L^{(n)} = (x \Trunx) \star \mu$. Due to Proposition \ref{BiasAbschProp}, Proposition \ref{LevyCLTProp} and Condition \ref{Cond1}\eqref{ObsSchCond4} the processes
$$\tilde Y_{\rho_{\alpha}}^{(n)} (t) = \sqrt{n \Delta_n} \left\{ \frac{1}{n \Delta_n} \sum \limits_{i=1}^n g_t^{(\alpha)}(\Deli L^{(n)}) - N_{\rho_{\alpha}}(t) \right\}$$
converge weakly to $\Gb_{\rho_{\alpha}}$ in $\linfr$. Thus it suffices to show
\begin{eqnarray}
\label{DiffKonv}
\frac{1}{\sqrt{n \Delta_n}} \sup \limits_{t \in \R} \left| \sum \limits_{i=1}^n \left\{ g_t^{(\alpha)} (\Deli X) \Truniv - g_t^{(\alpha)}(\Deli L^{(n)}) \right\} \right| \stackrel{\Prob}{\longrightarrow} 0.
\end{eqnarray}
We proceed similarly to Step 5 in the proof of Theorem 13.1.1 in \cite{JacPro12}. Recall the constants $\ell$ and $\tilde \ell$ of \eqref{Tildeelldef} in Condition \ref{Cond1}. Then we have 
\begin{eqnarray}
\label{elleqn}
1 < \ell < \frac{1}{2r \ovw} \wedge (1+ \epsilon) \quad \text{ and also } \quad \ell < \frac{2(p-1) \ovw -1}{2(r-1) \ovw} \text{ if } r >1.
\end{eqnarray}
We set further
\begin{eqnarray}
\label{unseqDefEq}
u_n = (v_n)^{\ell} \quad \text{ and } \quad F_n = \lbrace x \colon |x| > u_n \rbrace
\end{eqnarray}
as well as
\begin{eqnarray}
\label{SmandlajumpDef}
\tibigj &=& (x \mathtt 1_{F_n}(x)) \star \mu, \nonumber \\
\tibigjal &=& (x \mathtt 1_{F_n \cap \lbrace |x| \leq \alpha/4 \rbrace }(x)) \star \mu, \nonumber \\
N_t^n &=& (\mathtt 1_{F_n} \star \mu)_t, \nonumber \\
\tilde X_t^{\prime n} &=& X_t - \tibigj  \nonumber \\
   &=& X_0 + \int \limits_0^t b_s ds + \int \limits_0^t \sigma_s dW_s \nonumber \\
	 &+& (x \mathtt 1_{F_n^C}(x)) \star 
	     (\mu - \overline \mu)_t - (x \mathtt 1_{\lbrace |x| \leq 1 \rbrace \cap F_n}(x)) \star \overline \mu_t, \nonumber \\
A_i^n	&=& \lbrace | \Deli \tilde X^{\prime n} | \leq v_n/2 \rbrace \cap \lbrace \Deli N^n \leq 1 \rbrace.
\end{eqnarray}

 Let $m \in \N$ be the integer from Condition \ref{Cond1}\eqref{DrianDiffCond}. Then by Lemma 2.1.5 in \cite{JacPro12} we obtain for $1 \leq i \leq n$ and any $0 < \delta < 1$
\begin{align*}
\Eb & \left| \Deli \big ( x \mathtt 1_{F_n^C}(x) \big ) \star \big ( \mu - \overline \mu \big ) \right|^m  \\
                    &\leq K \bigg ( \Delta_n \int \limits_{\lbrace |x| \leq u_n \rbrace} |x|^m \nu(dx) + \Delta_n^{m/2} \bigg \{
										  \int \limits_{\lbrace |x| \leq u_n \rbrace} |x|^2 \nu(dx) \bigg \}^{m/2} \bigg ) \\
										&\leq K(\delta) \big ( \Delta_n^{1+(m-r-\delta) \ell \ovw} + \Delta_n^{m/2} \big ).
\end{align*}
Furthermore, $\overline \mu (ds,dx) = ds \otimes \nu(dx)$ yields for $1 \leq i \leq n$ and arbitrary $\delta >0$
\begin{align*}
\left| \Deli \big ( x \mathtt 1_{\lbrace |x| \leq 1 \rbrace \cap F_n}(x) \star \overline \mu \big ) \right| &= \Delta_n \bigg |
              \int \limits_{\lbrace u_n < |x| \leq 1 \rbrace} x \nu(dx) \bigg | \\
					&\leq \Delta_n u_n^{-(r+ \delta -1)_+} \int \limits_{\lbrace u_n < |x| \leq 1 \rbrace} |x|^{r+ \delta} \nu(dx) \\
					& \leq K(\delta) \Delta_n^{1 - \ell \ovw (r+ \delta -1)_+}.
\end{align*}
Moreover, from an application of H\"older's inequality and the Burkholder-Davis-Gundy inequalities (equation (2.1.32) on page 39 in \cite{JacPro12}) we obtain with $A$ being the upper bound of the coefficients in Condition \ref{Cond1}\eqref{DrianDiffCond} for $1 \leq i \leq n$:
\begin{align*}
\Eb \bigg | \int \limits_{(i-1) \Delta_n}^{i \Delta_n} b_s ds \bigg |^m \leq \Delta_n^m \Eb \bigg ( \frac{1}{\Delta_n} \int \limits_{(i-1) \Delta_n}^{i \Delta_n} \left| b_s \right|^m ds \bigg ) \leq \Delta_n^m \Eb A^m \leq K \Delta_n^m 
\end{align*}
and
\begin{align*}
\Eb \bigg | \int \limits_{(i-1) \Delta_n}^{i \Delta_n} \sigma_s dW_s \bigg |^m &\leq K \Delta_n^{m/2} \Eb 
     \bigg ( \frac{1}{\Delta_n} \int \limits_{(i-1) \Delta_n}^{i \Delta_n} \left| \sigma_s \right|^2 ds \bigg )^{m/2} \\
		&\leq K \Delta_n^{m/2} \Eb A^m \leq K \Delta_n^{m/2}.
\end{align*}
Additionally, $N^n$ is a Poisson process with parameter $\nu(F_n) \leq K(\delta) /u_n^{r+ \delta}$ for each $\delta >0$. Therefore, for any $1 \leq i \leq n$ and some $K(\delta)$ we have
$$\Prob(\Deli N^n \geq 2) \leq K(\delta) \Delta_n^{2-2(r + \delta) \ell \ovw}.$$
Let us now choose $\delta > 0$ in such a way that $1- \ell \ovw(r+ \delta - 1)_+ > \ovw$. Then, for $n$ large enough we have $\Delta_n^{1- \ell \ovw(r+ \delta - 1)_+} \leq K v_n$, and Markov inequality gives 
\begin{align}
\label{ProbAbscheqn}
\sum \limits_{i=1}^n \Prob ((A_i^n)^C) &\leq K(\delta) n \big \{ \Delta_n^{2 - 2(r+\delta) \ell \ovw} + \Delta_n^{1+(m-r- \delta) \ell \ovw - m\ovw} \nonumber \\
  &+ \Delta_n^{m/2 - m \ovw}  + \Delta_n^{m - m \ovw} \big \}.
\end{align}
From the choice of the constants we further have
\begin{align*}
2-2(r+ \delta) \ell \ovw \geq 2-2r (1+ \epsilon) \ovw
\end{align*}
and
\begin{align*}
\big (1+(m-r- \delta) \ell \ovw - m \ovw \big ) &\wedge \big ( m/2 - m \ovw \big ) \geq 1+ 2 \ovw,
\end{align*}
again for $\delta >0$ small enough. Thus the right hand side of \eqref{ProbAbscheqn} converges to zero for this choice of $\delta$, using Condition \ref{Cond1}\eqref{ObsSchCond4} and \eqref{ObsSchCond6}. 

Consequently, we have $\Prob(B_n) \rightarrow 1$ for the sets 
\begin{eqnarray}
\label{BnsetsDef}
B_n = \bigcap \limits_{i=1}^n A_i^n.
\end{eqnarray}
On $B_n$, and with $n$ large enough such that $v_n \leq \alpha/4$, one of the following mutually exclusive possibilities holds for $1 \leq i \leq n$:
\begin{itemize}
\item[(i)] $\Deli N^n =0$. \\
      Then we have $\left| \Deli X \right| = \left| \Deli \tilde X^{\prime n} \right| \leq v_n /2$ and there is no jump larger than 
			$v_n$ on the interval $((i-1) \Delta_n, i \Delta_n]$. Thus $g_t^{(\alpha)} (\Deli X) \Truniv =0 = g_t^{(\alpha)}(\Deli L^{(n)})$ holds for all $t \in \R$ and the summand in 
			\eqref{DiffKonv} vanishes.
\item[(ii)] $\Deli N^n =1$ and $\Deli \tibigj = \Deli \tibigjal \neq 0$. \\
      So the only jump in $((i-1) \Delta_n, i \Delta_n]$ (of absolute size) larger than $v_n$ is in fact not larger than $\alpha /4$, and because of $v_n \leq \alpha /4$
			we have $|\Deli X | \leq \alpha /2$. Thus, as in the first case, $g_t^{(\alpha)} (\Deli X) \Truniv =0 = g_t^{(\alpha)}(\Deli L^{(n)})$ is true for all $t \in \R$ and the
			summand in \eqref{DiffKonv} is equal to zero.
\item[(iii)]$\Deli N^n =1$ and $\Deli \tibigj \neq 0$, but $\Deli \tibigjal =0$. \\
      So the only jump in $((i-1) \Delta_n, i \Delta_n]$ larger than $v_n$ is also larger than $\alpha /4$. If we define $\hatbigjal = (x \mathtt 1_{\lbrace |x| > \alpha/4 \rbrace}) \star \mu$, we get
			\begin{eqnarray*}
			\Deli X &=& \Deli \tilde X^{\prime n} + \Deli \hatbigjal \quad \text{ and } \\ 
			g_t^{(\alpha)}(\Deli L^{(n)}) &=& \rho_{\alpha}(\Deli \hatbigjal) \mathtt 1_{\lbrace | \Deli \hatbigjal | > v_n \rbrace} 
			                                  \Indit(\Deli \hatbigjal)
			\end{eqnarray*}
\end{itemize}

Now obtain an upper bound for the term in \eqref{DiffKonv} on $B_n$, as soon as $v_n \leq \alpha/4$:
\begin{align*}
\frac{1}{\sqrt{n \Delta_n}} &\sup \limits_{t \in \R}  \left| \sum \limits_{i=1}^n \left\{ g_t^{(\alpha)} (\Deli X) \Truniv - 
             g_t^{(\alpha)}(\Deli L^{(n)}) \right\} \right| \\
\leq & \frac{1}{\sqrt{n \Delta_n}} \sup \limits_{t \in \R} \big| \sum \limits_{i=1}^n \bigg\{ \rho_{\alpha}(\Deli X) 
       \mathtt 1_{\lbrace | \Deli X| > v_n \rbrace} \Indit(\Deli X) - \rho_{\alpha}(\Deli \hatbigjal) \times \\
			&\times \mathtt 1_{\lbrace | \Deli \hatbigjal | > v_n \rbrace} \Indit(\Deli \hatbigjal) \bigg\} \mathtt 1_{\lbrace | \Deli \hatbigjal | > 
			 \alpha/4 \rbrace} \mathtt 1_{\lbrace | \Deli \tilde X^{\prime n} | \leq v_n/2 \rbrace} \big| \\
\leq & C_n + D_n,
\end{align*}
where we can substitute $\Deli X = \Deli \tilde X^{\prime n} + \Deli \hatbigjal$ in the second line and with
\begin{multline*}
C_n = \frac{K}{\sqrt{n \Delta_n}} \sup \limits_{t \in \R} \sum \limits_{i=1}^n \left| \Indit(\Deli \tilde X^{\prime n} + 
                                \Deli \hatbigjal) - \Indit(\Deli \hatbigjal) \right|  \times \\
																\times \mathtt 1_{\lbrace | \Deli \hatbigjal | > \alpha/4 \rbrace} 
																\mathtt 1_{\lbrace | \Deli \tilde X^{\prime n} | \leq v_n/2 \rbrace}
\end{multline*}
and
\begin{multline*}
D_n = \frac{1}{\sqrt{n \Delta_n}} \sum \limits_{i=1}^n \big| \rho_{\alpha} (\Deli \tilde X^{\prime n} + \Deli \hatbigjal)
                            \mathtt 1_{\lbrace |  \Deli \tilde X^{\prime n} + \Deli \hatbigjal| > v_n \rbrace} - \\
														- \rho_{\alpha}(\Deli \hatbigjal) \mathtt 1_{\lbrace | \Deli \hatbigjal| > v_n \rbrace} \big| 
														  \mathtt 1_{\lbrace | \Deli \tilde X^{\prime n} | \leq v_n/2 \rbrace}.
\end{multline*}
Here, $K>0$ denotes an upper bound for $\rho$. Because of $\Prob(B_n) \rightarrow 1$ it is enough to show $C_n \stackrel{\Prob}{\rightarrow} 0$ and $D_n \stackrel{\Prob}{\rightarrow} 0$
 in order to verify \eqref{DiffKonv} and to complete the proof of Lemma \ref{step1}. 

First we consider $D_n$. Let $g$ be either $\rho_{\alpha}$ or $\rho_{\alpha}^{\prime}$. Then there exists a constant $K >0$ which depends only on $\alpha$, such that we have for $x,z \in \R$ and $v>0$:
\begin{multline}
\label{ralralprdiabs}
\big| g(x+z) \mathtt 1_{\lbrace | x+z | >v \rbrace} - g(x) \mathtt 1_{\lbrace |x| > v \rbrace} \big| 
                                            \mathtt 1_{\lbrace |z| \leq v/2 \rbrace} \\
													\leq K (|x|^p \mathtt 1_{\lbrace |x| \leq 2v \rbrace} + |x|^{p-1} |z| \mathtt 1_{\lbrace |z| \leq v/2 \rbrace}).
\end{multline}
Note that for $|x+z| >v$ and $|x| >v$ we use the mean value theorem and $|z| \leq |x|$ as well as $\left| \frac{dg}{dx}(x) \right| \leq K |x|^{p-1}$ for all $x \in \R$ by the assumptions on $\rho$. In all other cases in which the left hand side does not vanish we have $|z| \leq |x| \leq 2v$ as well as $\left| g(x) \right| \leq K |x|^p$ for all $x \in \R$ by another application of the mean value theorem and the assumptions on $\rho$. 

Thus $D_n \stackrel{\Prob}{\rightarrow} 0$ holds, if we can show
\begin{align*} 
a_n = \frac{1}{\sqrt{n \Delta_n}} \sum \limits_{i=1}^n \Eb \left\{ \left| \Deli \hatbigjal \right|^p 
                        \mathtt 1_{\lbrace | \Deli \hatbigjal| \leq 2 v_n \rbrace} \right\} \rightarrow 0
\end{align*}
and
\begin{align*}
b_n = \frac{v_n}{2 \sqrt{n \Delta_n}} \sum \limits_{i=1}^n \Eb \big| \Deli \hatbigjal \big|^{p-1} \rightarrow 0. 
\end{align*}
For $y \in \R_+$ set
$$\widehat \delta_{\alpha}(y) = \int | x |^y \mathtt 1_{\lbrace |x| > \alpha/4 \rbrace} \nu(dx).$$
The assumptions on the L\'evy measure $\nu$ and $p> 2 \vee (1+ 3r)$ yield a constant $K >0$ with
\begin{align*}
\left\{ \widehat \delta_{\alpha} (r) \vee \widehat \delta_{\alpha}(1) \vee \widehat \delta_{\alpha}(p-1) \right\} \leq  K \quad \text{ and}  \quad |x|^p \mathtt 1_{\lbrace |x| \leq 2 v_n \rbrace} &\leq K v_n^{p-r} |x|^r.
\end{align*}
We obtain the desired result with Lemma 2.1.7 (b) in \cite{JacPro12} and Condition \ref{Cond1}\eqref{ObsSchCond4} as follows:
\begin{align*}
a_n &\leq \frac{K}{\sqrt{n \Delta_n}} n v_n^{p-r} \left\{ \Delta_n \widehat \delta_{\alpha}(r) + 
                 \Delta_n^{r \vee 1} \widehat \delta_{\alpha}(1)^{r \vee 1} \right\} = O \bigg (\sqrt{n \Delta_n^{1+ 2 \ovw (p-r)}} \bigg ) \\
     &= o \bigg ( \sqrt{n \Delta_n^{1+ 2 \ovw }} \bigg) \rightarrow 0, \\
b_n & \leq \frac{K}{\sqrt{n \Delta_n}} n v_n (\Delta_n \widehat \delta_{\alpha}(p-1) + 
           \Delta_n^{(p-1) \vee 1} \widehat \delta_{\alpha}(1)^{(p-1) \vee 1}) = O \bigg ( \sqrt{n \Delta_n^{1+ 2 \ovw}} \bigg ) \rightarrow 0.
\end{align*}
Finally, we show $C_n \stackrel{\Prob}{\rightarrow} 0$. Recall the L\'evy process of the large jumps, i.e.\
$$\hatbigjal = (x \mathtt 1_{\lbrace |x| > \alpha/4 \rbrace}) \star \mu,$$
and define for $1 \leq i,j \leq n$ with $i \neq j$ and the constant $\ovr$ of Condition \ref{Cond1}
\begin{align}
\label{RijnalDefEq}
R_{i,j}^{(n)} (\alpha) = \left\{ \big| \Deli \hatbigjal - \Delj \hatbigjal \big| \leq \Delta_n^{\ovr} \right\} \cap 
                       \left\{ \big| \Deli \hatbigjal \big| > \alpha/4 \right\} \cap B_n.
\end{align}
Let $x$ be arbitrary and either $y = 0$ or $|y| > \alpha/4$. Then, for $n$ large enough we have
$$\mathtt 1_{\lbrace |x - y| \leq \Delta_n^{\ovr} \rbrace} \mathtt 1_{\lbrace |x| > \alpha /4 \rbrace} \leq \mathtt 1_{\lbrace |x - y| \leq \Delta_n^{\ovr} \rbrace} \mathtt 1_{\lbrace |x| > \alpha /4 \rbrace} \mathtt 1_{\lbrace |y| > \alpha /4 \rbrace}.$$
Using the fact that on $B_n$ there is at most one jump of $\hatbigjal$ on an interval $((k-1) \Delta_n, k \Delta_n]$ with $1 \leq k \leq n$, we thus obtain
\begin{multline}
\label{PMRVorAbsch}
\Prob(R_{i,j}^{(n)}(\alpha)) \leq \int \int \int \int \int \mathtt 1_{\lbrace |x -y| \leq \Delta_n^{\ovr} \rbrace} 
                                \mathtt 1_{((j-1) \Delta_n, j \Delta_n]}(t) \mathtt 1_{\lbrace |y| > \alpha/4 \rbrace}  
                                \times \\
																\times \mathtt 1_{B_n}(\omega) \mu(\omega; dt,dy) \mathtt 1_{\lbrace |x| > \alpha /4 \rbrace} 
																\mathtt 1_{((i-1) \Delta_n, i \Delta_n]} (s) \mu(\omega; ds,dx) \Prob(d \omega).
\end{multline}
Now, forget about the indicator involving $B_n$ and  assume $j<i$. If $(\mathcal F_t)_{t \in \R_+}$ denotes the underlying filtration, the inner stochastic integral in \eqref{PMRVorAbsch} with respect to $\mu(\omega; dt,dy)$ is $\mathcal F_{j \Delta_n}$-measurable. Accordingly, the integrand in the stochastic integral with respect to $\mu(\omega; ds,dx)$ is in fact predictable. Fubini's theorem and the definition of the predictable compensator of an optional $\tilde{\mathcal P}$-$\sigma$-finite random measure
(Theorem II.1.8 in \cite{JacShi02}) yield for $n$ large enough:
\begin{align}
\label{PMengenRAbsch}
\Prob(R_{i,j}^{(n)}(\alpha)) \leq \Delta_n^2 \int \int \mathtt 1_{\lbrace |x-y| \leq \Delta_n^{\ovr} \rbrace} \mathtt 1_{\lbrace |x| > \alpha /4 \rbrace} \mathtt 1_{\lbrace |y| > \alpha /4 \rbrace} \nu(dy) \nu(dx).
\end{align}
We define the sets $J^{(1)}_n (\alpha)$ by their complements:
\begin{align}
\label{J1ijnalDefEq}
J^{(1)}_n(\alpha)^C = \bigcup \limits_{\stackrel{i,j =1}{i \neq j}}^n R_{i,j}^{(n)}(\alpha).
\end{align}
Then we have $\Prob (J^{(1)}_n(\alpha)) \rightarrow 1$, because \eqref{PMengenRAbsch}, Condition \ref{Cond1}\eqref{SeLevyDistCond} and Condition \ref{Cond1}\eqref{ObsSchCond3} show that there is a constant $K >0$ such that 
$$\Prob(J^{(1)}_n(\alpha)^C) \leq K n^2 \Delta_n^{2+q} \rightarrow 0.$$
So in order to obtain $C_n \stackrel{\Prob}{\rightarrow} 0$ we may assume that for $1 \leq i,j \leq n$ with $i \neq j$ 
$$\left| \Deli \hatbigjal - \Delj \hatbigjal \right| > \Delta_n^{\ovr}$$
is satisfied. But then for each $t \in \R$ at most $v_n / \Delta_n^{\ovr}$ summands in the sum of the definition of $C_n$ can be equal to $1$. We conclude
$$C_n = O \bigg (1/ \sqrt{n \Delta_n^{1+ 2(\ovr - \ovw)}} \bigg ) \rightarrow 0,$$
on $J^{(1)}_n(\alpha)$ by Condition \ref{Cond1}\eqref{ObsSchCond7}. \qed 

\medskip

\textbf{Proof of Proposition \ref{GalKonvLem}} Remark \ref{RescBMRem} reveals that the processes $\Gb_{f_n}$ are indeed the rescaled Brownian motions $\Bb_{f_n}(t) = \Bb(\int f_n^2(x) \mathtt 1_{(- \infty, t]}(x) \nu(dx))$ with a standard Brownian motion $\Bb$ on $[0,c]$ with $c = K^2 \int \big ( 1 \wedge x^4 \big ) \nu(dx)$. From \eqref{EquicontAbsch}, for each $\epsilon, \eta >0$ there exists some $\delta >0$ such that
\begin{eqnarray} \nonumber
\sup \limits_{n \in \N} \Prob^{\ast} \big ( \sup \limits_{d_{f_n}(u,v) < \delta} \left| \Bb_{f_n}(u) - \Bb_{f_n}(v) \right| > \epsilon \big ) \leq \Prob \big ( \sup \limits_{|u-v|^{1/2} < \delta} \left| \Bb (u) - \Bb(v) \right| > \epsilon \big ) < \eta,\\ \label{Equiabsch}
\end{eqnarray}
where $\Prob^{\ast}$ denotes outer probability, because for each $n \in \N$ the set on the left hand side is a subset of the set on the right hand side. But $d_{f_n}$ converges uniformly to $d_{f_0}$ by Lebesgue's convergence theorem under the given assumptions and therefore for each $\epsilon, \eta >0$ we have
$$\limsup \limits_{n \rightarrow \infty} \Prob^{\ast} \big ( \sup \limits_{d_{f_0}(u,v) < \delta /2} \left| \Bb_{f_n}(u) - \Bb_{f_n}(v) \right| > \epsilon \big ) < \eta$$
with $\delta>0$ from \eqref{Equiabsch}. Thus, $\Gb_{f_n}$ is asymptotically uniformly $d_{f_0}$-equicontinuous in probability. Furthermore, it is immediate to see that $(\R, d_{f_0})$ is totally bounded. Trivially, the marginals of $\Gb_{f_n}$ converge to $\Gb_{f_0}$, because these are centered multivariate normal distributions and their covariance functions converge again by Lebesgue's dominated convergence theorem. Therefore the desired result holds due to Theorem 1.5.4 and Theorem 1.5.7 in \cite{VanWel96}.
\qed

\medskip 

\textbf{Proof of Lemma \ref{step3}} For $\alpha > 0$ define the following empirical processes:
$$\tilde Y_{\rho^{\prime}_{\alpha}}^{(n)} (t) = \sqrt{n \Delta_n} \left\{ \frac{1}{n \Delta_n} \sum \limits_{i=1}^n g_t^{\prime 
(\alpha)}(\Deli L^{(n)}) - N_{\rho^{\prime}_{\alpha}}(t) \right\}.$$
For $n \rightarrow \infty$ these processes converge weakly in $\linfr$, that is
$$\tilde Y_{\rho^{\prime}_{\alpha}}^{(n)} \weak \Gb_{\rho^{\prime}_{\alpha}},$$
due to Proposition \ref{BiasAbschProp} and Proposition \ref{LevyCLTProp}. On the other hand we have weak convergence
$$\Gb_{\rho^{\prime}_{\alpha}} \weak 0$$
in $\linfr$ as $\alpha \rightarrow 0$, by Proposition \ref{GalKonvLem}. Therefore, by using the Portmanteau theorem (Theorem 1.3.4 in \cite{VanWel96}) twice, we obtain for arbitrary $\eta >0$:
\begin{align*}
\limsup \limits_{\alpha \rightarrow 0} \limsup \limits_{n \rightarrow \infty} \Prob \big ( \sup \limits_{t \in \R} \left| 
                    \tilde Y_{\rho^{\prime}_{\alpha}}^{(n)} (t) \right| \geq \eta \big ) \leq 
										\limsup \limits_{\alpha \rightarrow 0} \Prob \big ( \sup \limits_{t \in \R} \left| \Gb_{\rho^{\prime}_{\alpha}} (t)
										\right| \geq \eta \big ) =0.
\end{align*}
Thus it suffices to show
\begin{eqnarray*}
\limsup \limits_{n \rightarrow \infty} \Prob \big ( \sup \limits_{t \in \R} \left| V^{(n)}_{\alpha}(t) \right| > \eta \big ) = 0,
\end{eqnarray*}
for each $\eta >0$ and every $\alpha >0$ on a neighbourhood of $0$, where $V^{(n)}_{\alpha}$ denotes 
\begin{eqnarray}
\label{DifferDefEq}
V^{(n)}_{\alpha} (t) = \frac{1}{\sqrt{n \Delta_n}} \sum \limits_{i=1}^n \left\{ g_t^{\prime (\alpha)} (\Deli X) \Truniv - 
g_t^{\prime (\alpha)}(\Deli L^{(n)}) \right\}.
\end{eqnarray}
We will proceed similarly as in the proof of Lemma \ref{step1}. Therefore, we consider again the quantities defined in \eqref{unseqDefEq} and 
\eqref{SmandlajumpDef}, and with the same $\ell$. 

First of all let $\alpha >0$ be fixed for the following consideration. As we have seen prior to \eqref{BnsetsDef} the sets $B_n$ satisfy $\Prob(B_n) \rightarrow 1$. Furthermore, on the set $B_n$, and if $v_n \leq \alpha$, we have three mutually exclusive possibilities for $1 \leq i \leq n$:
\begin{enumerate}
\item[(i)] $\Deli N^n =0$. \\
      Then we have $\left| \Deli X \right| = \left| \Deli \tilde X^{\prime n} \right| \leq v_n /2$ and there is no jump larger than 
			$v_n$ on the interval $((i-1) \Delta_n, i \Delta_n]$. Thus $g_t^{\prime (\alpha)} (\Deli X) \Truniv = 0 = g_t^{\prime (\alpha)}(\Deli 
			L^{(n)})$ holds for all $t \in \R$ and the $i$-th summand in \eqref{DifferDefEq} vanishes.
\item[(ii)] $\Deli N^n =1$ and $\Deli \tibigj \neq 0$, but $\Deli \tibigjeial =0$. \\
      So the only jump in $((i-1) \Delta_n, i \Delta_n]$ larger than $v_n$ is also larger than $2 \alpha$. Because 
			$\left| \Deli \tilde X^{\prime n} \right| \leq v_n /2 \leq \alpha/2$ holds, we have 
			$\left| \Deli X \right| \geq \alpha$, and consequently $g_t^{\prime (\alpha)} (\Deli X) \Truniv = 0 = 
			g_t^{\prime (\alpha)}(\Deli L^{(n)})$ using the definition of $g_t^{\prime (\alpha)}$.
\item[(iii)] $\Deli N^n =1$ and $\Deli \tibigj = \Deli \tibigjeial \neq 0$. \\
      Here we can write 
			$$\Deli X = \Deli \tilde X^{\prime n} + \Deli \tibigjeial$$
			and
			$$g_t^{\prime (\alpha)}(\Deli L^{(n)}) = \rho_{\alpha}^{\prime}(\Deli \tibigjeial) 
			\mathtt 1_{\lbrace \left| \Deli \tibigjeial \right| > v_n \rbrace} \Indit(\Deli \tibigjeial).$$
\end{enumerate}

Therefore on $B_n$, and as soon as $v_n \leq \alpha$, we have with $V^{(n)}_{\alpha}$ as in \eqref{DifferDefEq}:
\begin{align*}
\sup \limits_{t \in \R} & \left| V^{(n)}_{\alpha}(t) \right| \leq \frac{1}{\sqrt{n \Delta_n}} \sup \limits_{t \in \R} \Big| 
      \sum \limits_{i=1}^n \bigg\{ \rho^{\prime}_{\alpha}(\Deli X) \mathtt 1_{\lbrace | \Deli X| > v_n \rbrace} \Indit(\Deli X) - \\
     &- \rho^{\prime}_{\alpha}(\Deli \tibigjeial)	\mathtt 1_{\lbrace | \Deli \tibigjeial | > v_n \rbrace} \Indit(\Deli \tibigjeial) \bigg\}
		     	\mathtt 1_{\lbrace | \Deli \tilde X^{\prime n} | \leq v_n/2 \rbrace} \Big| \\
\leq & \hat C_n(\alpha) + \hat D_n (\alpha) + \hat E_n(\alpha),
\end{align*}
where $\Deli X = \Deli \tilde X^{\prime n} + \Deli \tibigjeial$ can be substituted in the first line and with
\begin{multline*}
\hat C_n(\alpha) = \frac{K}{\sqrt{n \Delta_n}} \sup \limits_{t \in \R} \sum \limits_{i=1}^n \left| \Indit(\Deli \tilde X^{\prime n} + 
                                \Deli \tibigjeial) - \Indit(\Deli \tibigjeial) \right|  \times \\
																\times \mathtt 1_{\lbrace | \Deli \tibigjeial | > \Delta_n^{\ovv} \rbrace} 
																\mathtt 1_{\lbrace | \Deli \tilde X^{\prime n} | \leq v_n/2 \rbrace},
\end{multline*}
\begin{multline*}
\hat D_n(\alpha) = \frac{1}{\sqrt{n \Delta_n}} \sum \limits_{i=1}^n \big| \rho^{\prime}_{\alpha} (\Deli \tilde X^{\prime n} + 
                            \Deli \tibigjeial) \mathtt 1_{\lbrace |  \Deli \tilde X^{\prime n} + \Deli \tibigjeial| > v_n \rbrace} - \\
														- \rho^{\prime}_{\alpha}(\Deli \tibigjeial) \mathtt 1_{\lbrace | \Deli \tibigjeial| > v_n \rbrace} \big| 
														  \mathtt 1_{\lbrace | \Deli \tibigjeial | > \Delta_n^{\ovv} \rbrace} 
														  \mathtt 1_{\lbrace | \Deli \tilde X^{\prime n} | \leq v_n/2 \rbrace},
\end{multline*}
\begin{multline*}
\hat E_n(\alpha) = \frac{1}{\sqrt{n \Delta_n}} \sup \limits_{t \in \R} \Big| \sum \limits_{i=1}^n 
         \bigg\{ \rho^{\prime}_{\alpha}(\Deli X) \mathtt 1_{\lbrace | \Deli X| > v_n \rbrace} \Indit(\Deli X) - 
				\rho^{\prime}_{\alpha}(\Deli \tibigjeial) \times	\\
				\times \mathtt 1_{\lbrace | \Deli \tibigjeial | > v_n \rbrace} \Indit(\Deli \tibigjeial) \bigg\}
		     	\mathtt 1_{\lbrace | \Deli \tilde X^{\prime n} | \leq v_n/2 \rbrace}  
					\mathtt 1_{\lbrace | \Deli \tibigjeial | \leq \Delta_n^{\ovv} \rbrace}  \Big|
\end{multline*}
where $K>0$ denotes an upper bound for $\rho$  and $\ovv >0$ is the constant from Condition \ref{Cond1}\eqref{FiLevyDistCond}. In the first line of the definition of $\hat E_n (\alpha)$ we will later substitute
$\Deli X = \Deli \tilde X^{\prime n} + \Deli \tibigjeial$ again.

Now it suffices to show for each fixed $\alpha, \eta >0$: 
\begin{align}
\label{a} &\limsup_{n \rightarrow \infty} \Prob( \hat C_n(\alpha) > \eta) = 0,\\
\label{b} &\limsup_{n \rightarrow \infty} \Prob( \hat D_n(\alpha) > \eta) =0, \\
\label{c} &\limsup_{n \rightarrow \infty} \Prob( \hat E_n(\alpha) > \eta) = 0.
\end{align}
Concerning \eqref{a}, assume without loss of generality that $2 \alpha < \alpha_0$ with the constant $\alpha_0$ of Condition 
\ref{Cond1}\eqref{FiLevyDistCond}. Similar to \eqref{RijnalDefEq} we define for $1 \leq i,j \leq n$ with $i \neq j$ and the constants $\ovv < \ovr$ of Condition \ref{Cond1}:
$$S_{i,j}^{(n)}(\alpha) = \left\{ \left| \Deli \tibigjeial - \Delj \tibigjeial \right| \leq \Delta_n^{\ovr} \right\} \cap
\left\{ \left| \Deli \tibigjeial \right| > \Delta_n^{\ovv} \right\} \cap B_n.$$
The same considerations as for \eqref{PMRVorAbsch} and \eqref{PMengenRAbsch} yield
\begin{eqnarray*}
\Prob(S_{i,j}^{(n)}(\alpha)) &\leq& \Delta_n^2 \int \int \mathtt 1_{\lbrace | x-y| \leq \Delta_n^{\ovr} \rbrace} \mathtt 1_{\lbrace 
\Delta_n^{\ovv} /2 < |x| \leq \alpha_0 \rbrace} \mathtt 1_{\lbrace \Delta_n^{\ovv} /2 < |y| \leq \alpha_0 \rbrace} \nu(dx) \nu(dy) 
         \nonumber \\
         &\leq& K \Delta_n^{2+q},
\end{eqnarray*}
for $n$ large enough, because of Condition \ref{Cond1}\eqref{FiLevyDistCond} and $\ovv < \ovr$.

Just as in \eqref{J1ijnalDefEq} we define
$$(J_n^{(2)} (\alpha))^C = \bigcup \limits_{\stackrel{i,j =1}{i \neq j}}^n S_{i,j}^{(n)} (\alpha).$$
Therefore $\Prob (J_n^{(2)}(\alpha)) \rightarrow 1$ holds, and on $J_n^{(2)}(\alpha) \cap B_n$ we have
$$\left| \Deli \tibigjeial - \Delj \tibigjeial \right| > \Delta_n^{\ovr}$$
for each two indices $i,j$ for which the summand in $\hat C_n(\alpha)$ does not vanish. So for each $t \in \R$ at most $v_n / \Delta_n^{\ovr}$ summands in $\hat C_n(\alpha)$ can be equal to $1$, and we have $\hat C_n(\alpha) \rightarrow 0$ pointwise on $J^{(2)}_n(\alpha) \cap B_n$ by Condition \ref{Cond1}\eqref{ObsSchCond7}.

Now we discuss \eqref{b}. Because of \eqref{ralralprdiabs} there is a constant $K>0$ such that 
\begin{multline*}
\left| \rho_{\alpha}^{\prime} (x+z) \mathtt 1_{\lbrace |x+z| > v \rbrace} - \rho_{\alpha}^{\prime}(x) \mathtt 1_{\lbrace |x| > v \rbrace} 
\right| \mathtt 1_{\lbrace |z| \leq v/2 \rbrace} \\
\leq K ( |x|^{p-1} |z| \mathtt 1_{\lbrace |z| \leq v/2 \rbrace} + |x|^p \mathtt 1_{\lbrace |x| \leq 2v \rbrace}).
\end{multline*}
Therefore it suffices to verify
\begin{align} \label{A}
\lim_{n \rightarrow \infty} c_n(\alpha) = 0
\end{align}
      with
			$$c_n(\alpha) = \frac{1}{\sqrt{n \Delta_n}} \sum \limits_{i=1}^n \Eb \left\{ \left| \Deli \tibigjeial \right|^p 
			\mathtt 1_{\lbrace \left| \Deli \tibigjeial \right| \leq 2 v_n \rbrace} \right\}$$
			and
\begin{align} \label{B}
\lim_{n \rightarrow \infty} d_n(\alpha) =0
\end{align}
      for 
			$$d_n(\alpha) = \frac{v_n}{2 \sqrt{n \Delta_n}} \sum \limits_{i=1}^n \Eb \left| \Deli \tibigjeial \right|^{p-1}.$$

In order to show \eqref{A} let $\alpha >0$ be fixed. Define further for $y \in \R_+$
$$\widehat \delta_{n, \alpha}(y) = \int |x|^y \mathtt 1_{\lbrace u_n < |x| \leq 2 \alpha \rbrace} \nu(dx).$$
Obviously, for each small $\delta >0$ there is a constant $K(\delta)>0$ with
$$|x|^p \mathtt 1_{\lbrace |x| \leq 2 v_n \rbrace} \leq K(\delta) v_n^{p-((r+\delta) \wedge 1)} |x|^{(r+ \delta) \wedge 1}$$
and 
$$\widehat \delta_{n, \alpha}((r+ \delta) \wedge 1) \leq K(\delta) u_n^{-((r+ \delta)-1)_+}$$
by the assumptions on the L\'evy measure in Condition \ref{Cond1}. Thus Lemma 2.1.7(b) in \cite{JacPro12} shows
\begin{align*}
\Eb &\left\{ \left| \Deli \tibigjeial \right|^p \mathtt 1_{\lbrace \left| \Deli \tibigjeial \right| \leq 2 v_n \rbrace} \right\} \\
     &\leq K(\delta) v_n^{p-((r+\delta) \wedge 1)} \Eb \left\{ \left| \Deli \tibigjeial \right|^{(r+\delta) \wedge 1} \right\} \\
        &\leq K(\delta) v_n^{p-((r+\delta) \wedge 1)} \Delta_n \widehat \delta_{n, \alpha}((r + \delta) \wedge 1)  \\
				&\leq 
				K(\delta) \Delta_n^{1+ (p - ((r+ \delta) \wedge 1) - \ell ((r+\delta)-1)_+) \ovw}
\end{align*}
for all $1 \leq i \leq n$, and consequently for $\delta$ small enough
$$c_n(\alpha) \leq \begin{cases}
                   K(\delta) \sqrt{n \Delta_n^{1+ 2(p-r-\delta) \ovw}} = o \big (\sqrt{n \Delta_n^{1+2 \ovw}} \big ) \rightarrow 0, \quad \text{ if } r < 1 \\
									 K(\delta) \sqrt{n \Delta_n^2} \rightarrow 0, \quad \text{ if } r \geq 1
									\end{cases}$$
by Condition \ref{Cond1}\eqref{ObsSchCond4}. The second case holds because of $(p-1- \ell (r-1)) \ovw >\frac{1}{2}$ from \eqref{elleqn}.

In order to establish \eqref{B}, note from the assumptions on $\nu$ in Condition \ref{Cond1} and from $p-1 >r$ that we have for each small $\delta >0$ and $n \in \N$ large enough:
$$\widehat \delta_{n,\alpha} (1) \leq K(\delta) u_n^{-(r+\delta-1)_+} \quad \text{ and } \quad \widehat \delta_{n,\alpha} (p-1) \leq K(\delta).$$
Furthermore, Lemma 2.1.7 (b) in \cite{JacPro12} gives
$$\Eb \left| \Deli \tibigjeial \right|^{p-1} \leq K(\delta) ( \Delta_n \widehat \delta_{n, \alpha}(p-1) + \Delta_n^{(p-1) \vee 1} 
(\widehat \delta_{n, \alpha}(1))^{(p-1) \vee 1}),$$
for each $1 \leq i \leq n$. Therefore we obtain from Condition \ref{Cond1}\eqref{ObsSchCond4} and $p >2$, for $\delta$ small enough:
\begin{align*}
d_n(\alpha) &\leq K(\delta) \left\{ \frac{1}{\sqrt{n \Delta_n}} n \Delta_n v_n + \frac{1}{\sqrt{n \Delta_n}} n \Delta_n^{p-1} v_n 
              u_n^{-(r+\delta -1)_+ (p-1)} \right\} \\
						&\leq K(\delta) \left\{ \sqrt{n \Delta_n^{1+ 2 \ovw}} + \sqrt{ n \Delta_n^{2(p-1) - 2(p-1)(r+\delta-1)_+ \ell \ovw 
						      -1 + 2 \ovw}} \right\} \\
						&\leq \begin{cases}
						      K(\delta) \sqrt{n \Delta_n^{1+ 2 \ovw}} \rightarrow 0, \quad \text{ if } r < 1, \\
									K(\delta) \left\{ \sqrt{n \Delta_n^{1+ 2 \ovw}} + \sqrt{n \Delta_n^{2+2 \ovw}} \right\} \rightarrow 0, \quad \text{ if } r \geq 1,
									\end{cases}
\end{align*}
where the last case follows, because when $r \geq 1$ we have $p > 1+3r \geq 4$ and thus
$$2(p-1)  - 2(p-1)(r+ \delta -1) \ell \ovw -1 > 2(p-1)[1- \frac{1}{2}] -1 = p-1-1 >2,$$
for $0<\delta <1$ by the choice of $\ell < \frac{1}{2r \ovw}$. 

Concerning \eqref{c}, let $\alpha >0$ be fixed. Due to the indicator functions in the definition of $\hat E_n(\alpha)$ we have for $n \in \N$ large enough:
$$| \Deli \tibigjeial | \leq \Delta_n^{\ovv} \quad \text{ and } \quad |\Deli \tilde X^{\prime n} + \Deli \tibigjeial | \leq 2 \Delta_n^{\ovv},$$
because of $v_n = \gamma \Delta_n^{\ovw}$ and $\ovv < \ovw$ from Condition \ref{Cond1}\eqref{FiLevyDistCond}. Therefore the $i$-th summand in the definition of $\hat E_n(\alpha)$ is bounded by $K \Delta_n^{p \ovv}$ due to the behaviour of $\rho$ near zero. Thus
$$\hat E_n(\alpha) = O \bigg ( \sqrt{n \Delta_n^{2p \ovv -1}} \bigg ) \rightarrow 0,$$
for each $\omega \in \Omega$ by Condition \ref{Cond1}\eqref{ObsSchCond5}. \qed

\medskip

\textbf{Proof of Theorem \ref{ConvThm}} In order to establish weak convergence we use Theorem 1.12.2 in \cite{VanWel96}. With  $\Eb^{\ast}$ denoting outer expectation it is sufficient to prove
$$\Eb^{\ast} h(G_{\rho}^{(n)}) \rightarrow \Eb h(\Gb_{\rho})$$
for each bounded Lipschitz function $h \in \text{BL}_1(\linfr)$, that is $\| h \|_{\infty} \leq 1$ and $h$ is Lipschitz continuous with a Lipschitz constant bounded by $1$. Here, we use that the tight process $\Gb_{\rho}$ is also separable. 

Thus, let $h \in \text{BL}_1(\linfr)$ and $\delta >0$. Then we choose $\alpha >0$ with 
\begin{eqnarray}
\label{ProbAbschEq}
\limsup \limits_{n \rightarrow \infty} \Prob (\sup \limits_{t \in \R} |G_{\rho,n}^{\prime (\alpha)}(t)| > \delta/6) \leq \delta/13
\end{eqnarray}
and
\begin{eqnarray}
\label{ErwwHPrAbschEq}
\left| \Eb h(\Gb_{\rho_{\alpha}}) - \Eb h(\Gb_{\rho}) \right| \leq \delta/3.
\end{eqnarray}
\eqref{ProbAbschEq} is possible using Lemma \ref{step3}, and Lemma \ref{step2} allows \eqref{ErwwHPrAbschEq}. For this $\alpha >0$ choose an $N \in \N$ with 
\begin{eqnarray*}
\left| \Eb^{\ast} h(G_{\rho,n}^{(\alpha)}) - \Eb h(\Gb_{\rho_{\alpha}}) \right| \leq \delta/3,
\end{eqnarray*}
for $n \geq N$. This is possible due to Lemma \ref{step1}. Now, because of the previous inequalities and the Lipschitz property of $h$, we have for $n \in \N$ large enough:
\begin{multline*}
\left| \Eb^{\ast} h(G_{\rho}^{(n)}) - \Eb h(\Gb_{\rho}) \right| \leq \\
\leq \Eb^{\ast} \left| h(G_{\rho}^{(n)}) - h(G_{\rho,n}^{(\alpha)}) \right| + \left| \Eb^{\ast} h(G_{\rho,n}^{(\alpha)}) - \Eb h(\Gb_{\rho_{\alpha}}) \right| + \left| \Eb h(\Gb_{\rho_{\alpha}}) - \Eb h(\Gb_{\rho}) \right| \leq \delta.
\end{multline*}
\qed

\bigskip
\noindent
\textbf{Acknowledgements.}
%%The authors would like to thank two unknown referees and an Associate Editor for their constructive comments on an earlier version of this manuscript, which led to a substantial improvement of the paper.
This work has been supported by the Collaborative Research Center ``Statistical modeling of nonlinear dynamic processes'' (SFB 823, Teilprojekt A1) 
of the German Research Foundation (DFG) which is gratefully acknowledged.

\bibliographystyle{chicago}
\bibliography{bibliography}
\end{document}